\documentclass[10pt,reqno,oneside]{amsart}%%%这里可以调整10pt，变化页边距，大的公式可以放下，不露出来
\usepackage{booktabs}%????\toprule?\midrule?\bottomrule
\usepackage{multirow}%??????\multicolumn{}{}{}

\usepackage{graphicx}
\usepackage{amssymb}
\usepackage{epstopdf}
\usepackage{cite}
\usepackage[a4paper, total={6in, 9in}]{geometry}
\usepackage{amsmath,amsfonts,amsthm,mathrsfs,amssymb,cite,dsfont}
\usepackage[usenames]{color}
\usepackage{bm}
\usepackage{subfigure}

\usepackage{algorithm}% http://ctan.org/pkg/algorithms
\usepackage{algpseudocode}% http://ctan.org/pkg/algorithmicx
\algrenewcommand\algorithmicrequire{\textbf{Input:}}
\algrenewcommand\algorithmicensure{\textbf{Output:}}

\usepackage{multirow}
\makeatletter
\def\BState{\State\hskip-\ALG@thistlm}
\makeatother
\usepackage{cancel}

\usepackage{arydshln}%%%%绘制矩阵里的实线、虚线

\usepackage{amsmath,amsfonts,amsthm,mathrsfs,amssymb,cite,mathrsfs}
\usepackage[usenames]{color}

\newtheorem{thm}{Theorem}[section]

\newtheorem{lem}{Lemma}[section]
\newtheorem{prop}{Proposition}[section]
\theoremstyle{definition}
\newtheorem{defn}{Definition}[section]

\newtheorem{rem}{Remark}[section]

\theoremstyle{remark}

\numberwithin{equation}{section}
\numberwithin{equation}{section}

\newcounter{saveeqn}
\title[Sub-wavelength resonances of
elastic waves in 2D]{Sub-wavelength resonances in two-dimensional multi-layer elastic media}
\author{Yan Jiang}
\address{Department of Mathematics, City University of Hong Kong, Hong Kong SAR, China.}
\email{yjian24@cityu.edu.hk, jiangyan20@mails.jlu.edu.cn}

\author{Hongyu Liu}
\address{Department of Mathematics, City University of Hong Kong, Hong Kong SAR, China.}
\email{hongyu.liuip@gmail.com, hongyliu@cityu.edu.hk}

\author{Fanbo Sun}
\address{School of Mathematics and Statistics, Central South University, Changsha, Hunan, P. R. China.}
\email{fanbo\_sun2023@163.com}

\author{Yajuan Wang}
\address{Department of Mathematics, City University of Hong Kong, Hong Kong SAR, China.}
\email{yajuwang@cityu.edu.hk, wangyjsnu@163.com}

\usepackage{caption}
\date{} % Activate to display a given date or no date (if empty),
\begin{document}
\maketitle

\begin{abstract}
In this paper, we focus on the sub-wavelength resonances in two-dimensional elastic media characterized by high contrasts in both Lam\'e parameters and density. Our contributions are fourfold. First, it is proved that the operator $\hat{\mathbf{S}}_{\partial D}^{\omega}$, which serves as a leading order approximation to $\mathbf{S}_{\partial D}^{\omega}$ as $\omega\rightarrow0$, is invertible in the space $\mathcal{L}(L^{2}\left(\partial D)^{2},H^{1}(\partial D)^{2}\right)$. Second, based on layer potential techniques in combination with asymptotic analysis, we derive an original formula for the leading-order terms of sub-wavelength resonance frequencies, which are controlled by the determinant of the $3N \times 3N$ matrices. Specifically, there are $3N$ resonance frequencies within an $N$-nested layer structure. In addition, the scattering field exhibits an enhancement coefficient on the order of $\mathcal{O}(\omega^{-2})$ as the incident frequency $\omega$ approaches the resonance frequency. Third, by applying spectral properties to solve the corresponding eigenvalue problem, we compute the quantitative expressions for sub-wavelength resonance frequencies within a disk. Finally, some numerical experiments are provided to illustrate theoretical results and demonstrate the existence of the sub-wavelength resonance modes.
%%%%用两种方法计算的圆

\medskip

\noindent{\bf Keywords:} sub-wavelength resonances, high contrast metamaterials, resonators, multi-layer structures, layer potentials, spectrum

\noindent{\bf 2020 Mathematics Subject Classification:~35R30; 35B30; 35B34}%35R30; 35B30; 35B34
%35B30, 35C20, 35R30, 47G40
\end{abstract}

\section{Introduction}
\subsection{Mathematical setup}
Focusing first mainly on the mathematics, but not the physical applications, we begin by describing the mathematical formulation for our study. Consider the total displacement field $\mathbf{u}$ controlled by the following system
\begin{equation}\label{Lame system2}
\begin{cases}
\mathcal{L}_{\tilde{\lambda}, \tilde{\mu}} \mathbf{u}(\mathbf{x})+\omega^{2} \tilde{\rho} \mathbf{u}(\mathbf{x})=0, & \mathbf{x} \in {D}, \\ \mathcal{L}_{\lambda, \mu} \mathbf{u}(\mathbf{x})+\omega^{2} \rho \mathbf{u}(\mathbf{x})=0, & \mathbf{x} \in \mathbb{R}^{2} \backslash \bar{D}, \\ \left.\mathbf{u}(\mathbf{x})\right|_{-}=\left.\mathbf{u}(\mathbf{x})\right|_{+}, & \mathbf{x} \in \Gamma_j^{\pm},1\leq j\leq N, \\ \left.\partial_{\boldsymbol{\nu}} \mathbf{u}(\mathbf{x})\right|_{+}=\left. \partial_{\tilde{\boldsymbol{\nu}}} \mathbf{u}(\mathbf{x})\right|_{-},& \mathbf{x} \in \Gamma_j^{+},1\leq j\leq N, \\
\left.\partial_{\tilde{\boldsymbol{\nu}}} \mathbf{u}(\mathbf{x})\right|_{+}=\left.\partial_{\boldsymbol{\nu}} \mathbf{u}(\mathbf{x})\right|_{-}, & \mathbf{x} \in \Gamma_j^{-},1\leq j\leq N, \\
\mathbf{u}^{s}:=\mathbf{u}-\mathbf{u}^{i} \quad \text { satisfies the radiation condition. } &
\end{cases}
\end{equation}
Here, the entire nested-resonator can be represented as ${D}=\cup_{j=1}^N{D}_j$, where ${D}_j$ is the bounded doubly-connected domain in $\mathbb{R}^{2}$, lying between the interior boundary $\Gamma_j^-$ and the exterior boundary $\Gamma_j^+(1\leq j\leq N)$. Each $\Gamma_j^-$ surrounds $\Gamma_{j+1}^+$ and $\Gamma_j^+$ surrounds $\Gamma_j^-(1\leq j\leq N-1)$. Consequently, the host matrix material can be written by $\mathbb{R}^2\backslash \bar{D}=\cup_{j=0}^N\hat{D}_j,$ where $\hat{D}_j$ denotes the region of the gap between $\Gamma_j^{-}$ and $\Gamma_{j+1}^{+}(1\leq j\leq N-1)$. Let $\hat{D}_0$ and $\hat{D}_N$ be the unbounded domain with boundary $\Gamma_{1}^{+}$ and the bounded domain with boundary $\Gamma_N^-$, respectively.~(see Figure \ref{fig:multi} for a schematic illustration).

Moreover, the corresponding elastic parameters in the resonator and background medium satisfy the following relationship,
\begin{equation}\label{delta}
\tilde{\lambda}=\frac{1}{\delta} \lambda, \quad \tilde{\mu}=\frac{1}{\delta} \mu, \quad \tilde{\rho}=\frac{1}{\epsilon} \rho,
\end{equation}
where $\delta \ll 1$ and $\epsilon \ll 1$. Besides, we assume that the contrast $\tau$ satisfies
\begin{equation}\label{tau2}
\tau:=\sqrt{\delta / \epsilon}=\mathcal{O}(1).
\end{equation}
This high-contrast assumption is the cause of the underlying system's sub-wavelength resonance response and will be at the center of our subsequent analysis.
\begin{figure}[htbp]
  \centering
  \begin{minipage}[b]{0.45\textwidth}
    \includegraphics[width=0.8\textwidth]{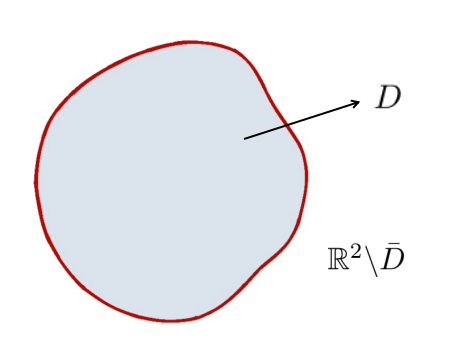}
   \caption{\label{fig:single}Single resonator.}
  \end{minipage}
\centering
  \begin{minipage}[b]{0.45\textwidth}
    \includegraphics[width=0.8\textwidth]{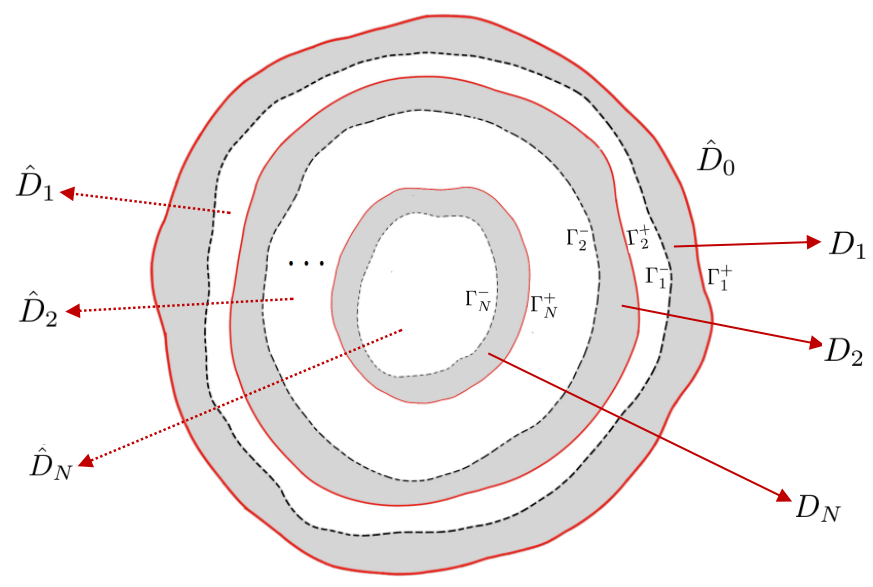}
   \caption{\label{fig:multi}$N$-nested resonators.}
  \end{minipage}
\end{figure}

In this paper, we are interested in studying the resonance behavior of the $N$-nested layer scatterers. Unlike the three-dimensional elastic case~(see \cite{Minnaertresonances,HLi2024}), the static single layer potential operator, denoted by $\mathbf{S}_{\partial D}$ in $\mathbb{R}^{2}$, may not be invertible, even for a disk. Moreover, both the operator and its inverse are not analytic with respect to the low-frequency. This introduces distinct mathematical challenges. In the following, we first provide the definition of the sub-wavelength resonance frequencies and resonance modes.
\begin{defn}
For $\delta > 0$, a sub-wavelength resonance frequency (eigenfrequency) $\omega=\omega(\delta)\in \mathbb{C}$ is defined to be such that
\begin{itemize}
\item[(i)] there exists a non-trivial solution to \eqref{Lame system2} with $\mathbf{u}^{i} = 0$, known as an associated resonance mode (eigenmode),
\item[(ii)] $\omega$ depends continuously on $\delta$ and satisfies $\omega \to 0$ as $\delta \to 0$.
\end{itemize}
\end{defn}
Our main result can be roughly summarised into the following theorem.
\begin{thm}\label{multifrequencies}
For $\delta\ll 1$, and $\epsilon \ll 1$, the system described by \eqref{Lame system2} has $3N$
sub-wavelength resonance frequencies~(counted with their multiplicities), denoted by $\omega_k~(1\leq k\leq3N)$. The leading-order term of these frequencies satisfies
\begin{align}\label{multiformula}
\det\Big(\rho\omega^2\ln\omega\hat{\mathbf{P}}+\rho\omega^2\left(\ln{(\sqrt{\rho}\tau)}\hat{\mathbf{P}}
+\hat{\mathbf{M}}\right)-\epsilon\hat{\mathbf{Q}}\Big)=0.
\end{align}
where the matrices $\hat{\mathbf{P}}$, $\hat{\mathbf{M}}$ and $\hat{\mathbf{Q}}$ are defined in \eqref{matriceshat1}. In fact, $\hat{P}_{ij}$, $\hat{M}_{ij}$ and $\hat{Q}_{ij}$ given in \eqref{matriceshat2} correspond to a diagonal block matrix and a tridiagonal block matrix, respectively.
\end{thm}
\begin{thm}\label{pointscattererapproximation}
For $\delta\ll 1$, and $\epsilon\ll 1$, the displacement filed~$\mathbf{u}_{{D}_j}$~within the resonator~${D}_j$~admits the following expression for $\omega\ll1$,
\vspace{-7pt}%%%%%设置行间距
$$
\mathbf{u}_{{D}_j}=\sum_{i=1}^3\varrho_i \boldsymbol{\xi}_i+\mathcal{O}(1),\ 1\leq j\leq N,
\vspace{-2pt}
$$
%\vspace{-2pt}
when $\omega$ is located in the following different regimes, the coefficients $\varrho_i$, satisfy
\vspace{1pt}
\begin{itemize}
\item[(i)] for $\omega$ is located far from the resonance frequencies $\omega_k$, the coefficients $\varrho_i=\mathcal{O}(1)$;
\vspace{2pt}
\item[(ii)] for $\omega$ coincides with the roots of \eqref{multiformula}, the coefficients $\varrho_i=\mathcal{O}(\omega^{-2})$.
\end{itemize}
Moreover, for $\omega^2\gg \delta$, there holds
\vspace{-7pt}
$$
{\mathbf{u}_{{D}_j}=\sum_{i=1}^3\varrho_i \boldsymbol{\xi}_i(1+{o}(1))}, \ 1\leq j\leq N,
\vspace{-2pt}
$$
%\vspace{2pt}
with the coefficients $\varrho_i=o(1)$.
\end{thm}
We will establish the details for Theorem \ref{multifrequencies} and Theorem \ref{pointscattererapproximation} to hold in Section 3.
\subsection{Background motivation and technical developments}The recent development shows that metamaterials with negative or high contrast parameters offer new possibilities for imaging and for the control of waves at deep sub-wavelength scales. Sub-wavelength resonators, such as air bubbles exhibiting Minnaert resonance \cite{1933Minnaert,Minnaertresonances}, have become fundamental in the development of phononic crystals, where periodic arrangements of these high-contrast resonators give rise to distinctive wave propagation phenomena \cite{Pendry,bandgapopening,phononiccrystals}, while also providing insights into the mechanisms of wave localization in disordered, finite systems \cite{skineffect1,defect1,Robustedge,DisorderedSystems}.

Mathematically, most existing research on metamaterials, which are designed by inducing specific resonances, has primarily concentrated on single- and double-layer structures. These include the surface plasmon resonance of metal nanoparticles, typically analyzed using the Drude model, and the core-shell geometry utilized in invisibility cloaks, which is achieved through anomalous local resonance \cite{2018Kangelasticity,Kangelasticity,SpectralTheory}. However, these structures are often limited by poor filtering performance and narrow bandgap width. In contrast, multi-layer scatterers exhibit multiple vibration modes and greater degrees of freedom, positioning them as promising candidates for overcoming these limitations. Indeed, Ammari et al. \cite{GPT1,GPT2,GPT3,SC1,SC2} proposed vanishing structures with generalized polarization tensors~(GPTs)~and a scattering coefficient~(SC), designing the structure using multi-coated concentric disks or balls and proving that near cloaking effects could be enhanced. To support multiband functionality and facilitate cloaking of larger objects, a multi-layer plasmon hybridization model was introduced, offering increased flexibility to suppress higher-order scattering modes \cite{Multifrequency,Invisibility}. Initial analyses focused on three- and four-layer structures \cite{Dandf1,Kulkarni,Plasmonhybridization}. Experimental and numerical studies \cite{multi5,multi7,multi8,multi9} have shown that multi-layer concentric radial resonators can generate multiple local resonance bandgaps in the sub-wavelength regime. Specifically, plasmon resonances in multi-layered concentric spheres and confocal ellipses were investigated in \cite{Dandf2,Kmodesplitting,Ksubwavelength,Wsubwavelength,Kbandgap}, revealing how plasmon modes interact and hybridize, generating localized surface plasmons at each interface. These studies demonstrate how mode splitting is influenced by geometry truncation and the breaking of rotational symmetry, thereby enabling cloaks with broader bandwidths and multiple operational frequencies. Despite substantial progress, mathematical modeling of sub-wavelength resonances and mode splitting mechanisms remains limited.

Inspired by the fact that plasmon modes in multi-layer systems interact and hybridize across interfaces, this paper
is to establish a sub-wavelength resonance theory for two-dimensional elastic media with high contrast. The coupling of longitudinal and transverse waves, along with the vectorial nature of the displacement field, renders the study of elastic sub-wavelength resonators notably more complex. The three-dimensional static single layer potential operator is always invertible and analytic with respect to the low-frequency regime, making it an important starting point for investigation. By using layer potential techniques and the variational method, papers \cite{HLi119141,variationalmethod,HLi2024,HLiLxu2024} provide mathematical proofs of monopolar, dipole and quadrupole resonances using soft elastic materials~(HISE structures)~and dimers (THES structures). For phononic crystals composed of periodically arranged resonators, the authors in \cite{ElasticBandgaps} demonstrate the existence of sub-wavelength band gaps. By contrast, in two-dimensional, similar to the derivations of Minnaert resonances given in Appendix B of \cite{Minnaertresonances}, elastic waves have the following four main challenges.

(1)~the single layer potential $\mathbf{S}_{\partial D}$ may not be invertible (even for a disk) in two dimensions, while this property always holds in three dimensions;

(2)~there is a logarithmic singularity in the asymptotic expansion of the single layer potential $\mathbf{S}_{\partial D}^{\omega}$ for small $\omega$, which means that it contains a double series comprising terms of the form $\omega^{2j}$ and $\omega^{2j}\ln\omega$ for $j\in\mathbb{N}$;

(3)~even though ${\mathbf{S}}_{\partial D}$ may not be invertible, it is necessary to prove the invertibility of leading term $\hat{\mathbf{S}}_{\partial D}^{\omega}$ in $\mathcal{L}(L^{2}\left(\partial D)^{2},H^{1}(\partial D)^{2}\right)$ for sufficient small $\omega$;

(4)~the spectral properties of the associated layer potential operators in two dimensions are more complex for radial geometry, involving not only the intrinsic coupling of all shear and compression waves, but also the corresponding Bessel and Hankel functions containing singular terms $\ln{k_s}$ or $\ln{k_p}$.

To establish the primary conclusion of this paper, we first construct the low-frequency asymptotic expansions for layer potential operator. It is worth mentioning that the corresponding space for the static Lam\'e system with Neumann boundary conditions is three-dimensional, which affect the invertibility of $\hat{\mathbf{S}}_{\partial D}^{\omega}$. In fact, we have obtained the proof using the relationship matrix ${\mathbf{C}}$ between the kernel spaces of operators $-\mathbf{I}_2 / 2+\mathbf{K}_{\partial D}$ and $-\mathbf{I}_2 / 2+\mathbf{K}_{\partial D}^{*}$. Next, utilizing layer potential techniques and Gohberg-Sigal theory, we analyzed the resonance modes of a single resonator for simplicity. Furthermore, the strict extension of resonance theory to multi-layer elastic structures has advanced the existing theoretical framework. The results show that the leading-order terms of the resonance frequencies are controlled by the determinant of the $3N \times 3N$ matrices, with the number of resonance frequencies increasing as the number of resonators grows. Specifically, there are $3N$ resonance frequencies within an $N$-nested layer structure. However, in two dimensions, only an implicit expression for the leading term can be derived. Therefore, we focus on the resonators with radial geometry, where explicit expressions for the sub-wavelength resonance frequencies within a disk are obtained by solving the corresponding eigenvalue problem. Additionally, as the incident frequency $\omega$ approaches the resonance frequency, the enhancement coefficient inside the resonator is on the order of  $\mathcal{O}(\omega^{-2})$. Finally, our results are validated through several numerical examples.

The remainder of this paper is organized as follows. Section 2 reviews main results on layer potentials, and proves the invertibility of the operator $\hat{\mathbf{S}}_{\partial D}^{\omega}$ in the space $\mathcal{L}(L^{2}\left(\partial D)^{2},H^{1}(\partial D)^{2}\right)$. In Section 3, we establish the primary conclusions, including the control equation for the sub-wavelength resonance frequency of multi-layer structures and the point scatterer approximation. Section 4 applies two methods to obtain a specific expression for the resonance frequency in radial geometry. Section 5 presents some numerical experiments to validate the theoretical results.

\section{Preliminaries}
\subsection{Layer potential theory}
We first briefly review the potential theory of the Lam\'e system. The fundamental solution $\boldsymbol{\Gamma}^{\omega}$ of the operator $\mathcal{L}_{\lambda, \mu}+\omega^{2}\rho$ in two dimensions is given by \cite{HAmmariElasticityImaging}
\begin{equation}\label{fundamental solution}
\boldsymbol{\Gamma}^{\omega}(\mathbf{x})=-\frac{\mathrm{i}}{4\mu}H_0^{(1)}\left(k_s|\mathbf{x}|\right)\mathbf{I}_2+\frac{\mathrm{i}}{4\omega^2\rho}\nabla\nabla\left(H_0^{(1)}\left(k_p|\mathbf{x}|\right)-H_0^{(1)}\left(k_s|\mathbf{x}|\right)\right),
\end{equation}
where $\nabla\nabla$ is the standard double differentiation, $\mathbf{I}_2$ stands for the $2\times2$ identity matrix, and
\begin{equation}\label{kskp}
k_{s}=\omega / c_{s}, \quad k_{p}=\omega / c_{p},
\end{equation}
with
\begin{equation}\label{cscp}
c_{s}=\sqrt{\mu / \rho}, \quad c_{p}=\sqrt{(\lambda+2 \mu) / \rho} .
\end{equation}
In particular, when $\omega=0$, we denote $\boldsymbol{\Gamma}^{0}$ by $\boldsymbol{\Gamma}$ for simplicity, and $\boldsymbol{\Gamma}$ has the following expression
\begin{equation}\label{gamma}
\boldsymbol{\Gamma}(\mathbf{x})=\alpha_1\ln|\mathbf{x}|\mathbf{I}_2-\alpha_2\frac{\mathbf{x}\mathbf{x}^T}{|\mathbf{x}|^2},
\end{equation}
where
\begin{equation}\label{alpha12}
\alpha_1=\frac{1}{4\pi}\left(\frac{1}{\mu}+\frac{1}{2 \mu+\lambda}\right) \quad \text { and } \quad \alpha_2=\frac{1}{4\pi}\left(\frac{1}{\mu}-\frac{1}{2 \mu+\lambda}\right) .
\end{equation}
Let $D$ be a bounded domain in $\mathbb{R}^{2}$ with connected Lipschitz boundaries. Then the single-layer potential associated with the fundamental solution $\boldsymbol{\Gamma}^{\omega}$ is defined by
\begin{equation}
\mathbf{S}_{\partial D}^{\omega}[\boldsymbol{\varphi}](\mathbf{x})=\int_{\partial D} \boldsymbol{\Gamma}^{\omega}(\mathbf{x}-\mathbf{y}) \boldsymbol{\varphi}(\mathbf{y}) d s(\mathbf{y}), \quad \mathbf{x} \in \mathbb{R}^{2},
\end{equation}
for $\boldsymbol{\varphi} \in L^{2}(\partial D)^{2}$. On the boundary $\partial D$, the conormal derivative of the single-layer potential satisfies the following jump formula
\begin{equation}\label{jump formula}
\left.\partial_{\boldsymbol{\nu}} \mathbf{S}_{\partial D}^{\omega}[\boldsymbol{\varphi}]\right|_{ \pm}(\mathbf{x})=\left( \pm \frac{1}{2} \mathbf{I}_2+\mathbf{K}_{\partial D}^{\omega, *}\right)[\boldsymbol{\varphi}](\mathbf{x}), %\quad \mathbf{x} \in \partial D,
\end{equation}
where
$$
\mathbf{K}_{\partial D}^{\omega, *}[\boldsymbol{\varphi}](\mathbf{x})=\text { p.v. } \int_{\partial D} \partial_{\boldsymbol{\nu}_{\mathbf{x}}} \boldsymbol{\Gamma}^{\omega}(\mathbf{x}-\mathbf{y}) \boldsymbol{\varphi}(\mathbf{y}) d s(\mathbf{y}),
$$
with p.v. standing for the Cauchy principal value. We would like to mention that the operator $\mathbf{K}_{\partial D}^{\omega, *}$ in \eqref{jump formula} is called the Neumann-Poincar\'e (N-P) operator, which is a critical operator in the investigation of metamaterials in elasticity. In what follows, we denote $\mathbf{S}_{\partial D}^{0}, \mathbf{K}_{\partial D}^{0, *}$ by $\mathbf{S}_{\partial D}, \mathbf{K}_{\partial D}^{*}$, respectively, for simplicity.
%For the further analysis, we present the asymptotic expansion of the fundamental solution $\boldsymbol{\Gamma}^{\omega}$ defined in \eqref{fundamental solution}.
\subsection{Auxiliary results of asymptotic analysis}
\begin{lem}\label{lemgamma}
Suppose $\omega\in \mathbb{R}_{+}$ and $\omega\ll1$. There holds the following asymptotic expansion% for $\mathbf{x}\in\mathbb{R}^{2}$,
\begin{equation}
\boldsymbol{\Gamma}^{\omega}(\mathbf{x})=\sum_{n=0}^\infty\Big(\omega^{2n}\boldsymbol{\Gamma}_{n}^{1}(\mathbf{x})
+\omega^{2n-2}\boldsymbol{\Gamma}_{n}^{2}(\mathbf{x})
+\omega^{2n}\ln{\omega}\boldsymbol{\Gamma}_{n}^{3}(\mathbf{x})
+\omega^{2n-2}\ln{\omega}\boldsymbol{\Gamma}_{n}^{4}(\mathbf{x})\Big),
\end{equation}
where
$$
\boldsymbol{\Gamma}_{n}^{1}(\mathbf{x})=\frac{1}{\mu}\frac{1}{c_{s}^{2n}}\Big(b_{n}\ln{|\mathbf{x}|}|\mathbf{x}|^{2n}\mathbf{I}_2+(c_{n}-b_{n}\ln{c_s})|\mathbf{x}|^{2n}\mathbf{I}_2\Big),
\quad\quad \boldsymbol{\Gamma}_{n}^{3}(\mathbf{x})=\frac{1}{\mu}\frac{b_{n}}{c_{s}^{2n}}|\mathbf{x}|^{2n}\mathbf{I}_2,\nonumber
$$
\begin{align}
\boldsymbol{\Gamma}_{n}^{2}(\mathbf{x})=&\frac{1}{\rho}2nb_n\big(\frac{1}{c_{s}^{2n}}-\frac{1}{c_{p}^{2n}}\big)\ln{|\mathbf{x}|}|\mathbf{x}|^{2n-2}\mathbf{I}_2\nonumber\\
&+\frac{1}{\rho}\Big((b_n+2nc_n)\big(\frac{1}{c_{s}^{2n}}-\frac{1}{c_{p}^{2n}}\big)-2nb_n\big(\frac{\ln{c_s}}{c_{s}^{2n}}-\frac{\ln{c_p}}{c_{p}^{2n}}\big)\Big)|\mathbf{x}|^{2n-2}\mathbf{I}_2\nonumber\\
&+\frac{1}{\rho}2n(2n-2)b_n\big(\frac{1}{c_{s}^{2n}}-\frac{1}{c_{p}^{2n}}\big)\ln{|\mathbf{x}|}|\mathbf{x}|^{2n-4}\mathbf{x}\mathbf{x}^T\nonumber\\
&+\frac{1}{\rho}\Big[\Big((4n-2)b_n+2n(2n-2)c_n\Big)\big(\frac{1}{c_{s}^{2n}}-\frac{1}{c_{p}^{2n}}\big)-2n(2n-2)b_n\big(\frac{\ln{c_s}}{c_{s}^{2n}}-\frac{\ln{c_p}}{c_{p}^{2n}}\big)\Big]|\mathbf{x}|^{2n-4}\mathbf{x}\mathbf{x}^T,\nonumber
\end{align}
and
$$
\boldsymbol{\Gamma}_{n}^{4}(\mathbf{x})=\frac{1}{\rho}2nb_n\big(\frac{1}{c_{s}^{2n}}-\frac{1}{c_{p}^{2n}}\big)\Big(|\mathbf{x}|^{2n-2}\mathbf{I}_2+(2n-2) |\mathbf{x}|^{2n-4}\mathbf{x}\mathbf{x}^T\Big).
$$
Here, the parameters $b_0=\frac{\pi}{2}$, $c_0=\frac{\pi}{4}E_c$ with $E_c=2\gamma -\mathrm{i}\pi-2\ln2$, $\gamma$ being Euler's constant, and
\begin{equation}\label{bncn}
b_n=\frac{(-1)^n}{2\pi}\frac{1}{2^{2n}(n!)^2},\quad c_n=b_n\big(\gamma-\ln2-\frac{\mathrm{i}\pi}{2}-\sum_{j=1}^{n}\frac{1}{j}\big), \quad n\geq 1.%\nonumber
\end{equation}
%and $c_{s}$, $c_{p}$ defined in \eqref{cscp}, $\gamma$ is Euler's constant.
\end{lem}
\begin{proof}
From \cite{Minnaertresonances}, we have, for $k\ll1$,
\begin{equation}\label{H01}
-\frac{\mathrm{i}}{4}H_0^{(1)}(k|\mathbf{x}|)=\sum_{n=0}^{+\infty}c_n(k|\mathbf{x}|)^{2n}+\sum_{n=0}^{+\infty}b_n\ln(k|\mathbf{x}|)(k|\mathbf{x}|)^{2n},
\end{equation}
where $b_{n}$, $c_{n}$ are defined in \eqref{bncn}. By direct calculations, it holds that
\begin{align}\label{GH01}
&\nabla\nabla\Big(-\frac{\mathrm{i}}{4}H_0^{(1)}(k|\mathbf{x}|)\Big)=\sum_{n=0}^\infty(b_n+2nc_n)k^{2n}|\mathbf{x}|^{2n-2}\mathbf{I}_2+\sum_{n=0}^\infty2nb_nk^{2n}\ln(k|\mathbf{x}|)|\mathbf{x}|^{2n-2}\mathbf{I}_2\nonumber\\
&+\sum_{n=0}^\infty\Big((4n-2)b_n+2n(2n-2)c_n\Big)k^{2n}|\mathbf{x}|^{2n-4}\mathbf{x}\mathbf{x}^T\nonumber\\
&+\sum_{n=0}^\infty2n(2n-2)b_nk^{2n}\ln(k|\mathbf{x}|)|\mathbf{x}|^{2n-4}\mathbf{x}\mathbf{x}^T.
\end{align}
The proof can be completed by substituting equations \eqref{H01} and \eqref{GH01} into \eqref{fundamental solution}.
\end{proof}
\begin{prop}\label{fundamental solution expansion}
The fundamental solution defined in \eqref{fundamental solution} has the following asymptotic expansion for $\omega\ll1$,
$$
\boldsymbol{\Gamma}^{\omega}(\mathbf{x})=\boldsymbol{\Gamma}(\mathbf{x})+\gamma_{\omega}\mathbf{I}_2+\omega^2\ln\omega\rho\boldsymbol{\Gamma_1}(\mathbf{x})
+\omega^2\rho\ln{\sqrt{\rho}}\boldsymbol{\Gamma_1}(\mathbf{x})
+\omega^2\rho\boldsymbol{\Gamma_2}(\mathbf{x})+ \mathcal{O}(\omega^4\ln{\omega}+\omega^4),
$$
where
$$
\boldsymbol{\Gamma_1}(\mathbf{x})=\beta_2|\mathbf{x}|^{2}\mathbf{I}_2
+\beta_3\mathbf{x}\mathbf{x}^T,\quad \quad \boldsymbol{\Gamma_2}(\mathbf{x})=\beta_1|\mathbf{x}|^{2}\mathbf{I}_2+\beta_2\ln|\mathbf{x}||\mathbf{x}|^{2}\mathbf{I}_2
+\beta_3\ln|\mathbf{x}|\mathbf{x}\mathbf{x}^{T}+\beta_4\mathbf{x}\mathbf{x}^{T},
$$
and $\gamma_{\omega}=\ln (\sqrt{\rho}\omega) \alpha_1+\alpha$ with
$$ \alpha=\frac{\alpha_1}{2}E_c+\frac{\alpha_2}{2}-\frac{1}{8\pi}\left(\frac{1}{\mu}\ln\mu+\frac{1}{2\mu+\lambda}\ln{(2\mu+\lambda)}\right).
$$
Moreover,
\begin{align}
\beta_1:&=\left(\frac{1}{2}E_c-1\right)\beta_2-\frac{1}{8}\beta_3+\frac{1}{2^6\pi}\left(\frac{3}{\mu^2}\ln{\mu} +\frac{1}{(2\mu+\lambda)^2}\ln{(2\mu+\lambda)}\right),\nonumber
\end{align}
$$
\beta_2:=-\frac{1}{2^5\pi}\left(\frac{3}{\mu^2}+\frac{1}{(2\mu+\lambda)^2}\right),\quad \quad \beta_3:=\frac{1}{2^4\pi}\left(\frac{1}{\mu^2}-\frac{1}{(2\mu+\lambda)^2}\right),
$$
and
$$
\beta_4:=\frac{1}{4}\left(2E_c-3\right)\beta_3
-\frac{1}{2^5\pi}\left(\frac{1}{\mu^2}\ln{\mu} -\frac{1}{(2\mu+\lambda)^2}\ln{(2\mu+\lambda)}\right).
$$
%with $\alpha_{1}$, $\alpha_{2}$ given in \eqref{gamma} and \eqref{alpha12}, respectively.
\end{prop}
\begin{proof}
The proof follows from the asymptotic expansion of the function $\boldsymbol{\Gamma}^{\omega}$ in Lemma \ref{lemgamma} and tedious calculations.
\end{proof}
%Furthermore, we provide the following Lemma.
\begin{prop}\label{fundamental tilde solution expansion}
The fundamental solution defined in \eqref{fundamental solution}, with parameters $(\tilde{\lambda}, \tilde{\mu},\tilde{\rho})$, has the following asymptotic expansion for $\omega\ll1$,
\begin{align}
\tilde{\boldsymbol{\Gamma}}^{\omega}(\mathbf{x})&=\delta\boldsymbol{\Gamma}(\mathbf{x})+\delta\gamma_{\tau\omega}\mathbf{I}_2
+\delta\omega^2\ln\omega\rho\tau^2\boldsymbol{\Gamma_1}(\mathbf{x})+\delta\omega^2\rho\tau^2 \ln(\sqrt{\rho}\tau)\boldsymbol{\Gamma_{1}}(\mathbf{x})\nonumber \\
&+\delta\omega^2\rho\tau^2\boldsymbol{\Gamma_2}(\mathbf{x})
+\mathcal{O}(\delta\tau^4\omega^4\ln{\omega}+\delta\tau^4\omega^4),\nonumber
\end{align}
where $\gamma_{\tau\omega}=\ln (\sqrt{\rho}\tau\omega) \alpha_1+\alpha$.
\end{prop}
\begin{proof}
As Proposition \ref{fundamental solution expansion}, summing relations \eqref{delta} and \eqref{tau2} provides an proof.
\end{proof}
\begin{lem}\label{calculatelame}
The following identities hold
\begin{equation}
\mathcal{L}_{\lambda,\mu} \boldsymbol{\Gamma_1}(\mathbf{x})=a_{\lambda,\mu}\mathbf{I}_2,\quad\quad
\mathcal{L}_{\lambda,\mu} \boldsymbol{\Gamma_2}(\mathbf{x})=b_{\lambda,\mu}\mathbf{I}_2-\boldsymbol{\Gamma}(\mathbf{x}),\nonumber
\end{equation}
where the constants
\begin{equation}\label{lamea}
a_{\lambda,\mu}:=(2\lambda+6\mu)\beta_2+(3\lambda+5\mu)\beta_3,
\end{equation}
and
\begin{equation}\label{lameb}
b_{\lambda,\mu}:=\beta_{1}(2\lambda+6\mu)+\beta_{2}(\lambda+5\mu)+\beta_{3}(\lambda+\mu)+\beta_{4}(3\lambda+5\mu),
\end{equation}
with $\beta_{1}$, $\beta_{2}$ and $\beta_{3}$, $\beta_{4}$ defined in Proposition \ref{fundamental solution expansion}.
\end{lem}
\begin{proof}
By direct calculations, there exists the following facts
$$
\mathcal{L}_{\lambda,\mu}\left(|\mathbf{x}|^{2}\mathbf{I}_2\right)=(2\lambda+6\mu)\mathbf{I}_2,\quad \quad \quad \mathcal{L}_{\lambda,\mu}\left(\mathbf{x}\mathbf{x}^{T}\right)=(3\lambda+5\mu)\mathbf{I}_2,
$$
$$
\mathcal{L}_{\lambda,\mu}\left(\ln|\mathbf{x}||\mathbf{x}|^{2}\mathbf{I}_2\right)=(\lambda+5\mu)\mathbf{I}_2+(2\lambda+6\mu)\ln|\mathbf{x}|\mathbf{I}_2
+(2\lambda+2\mu)\frac{\mathbf{x}\mathbf{x}^{T}}{|\mathbf{x}|^2},
$$
and
$$
\mathcal{L}_{\lambda,\mu}\left(\ln|\mathbf{x}|\mathbf{x}\mathbf{x}^{T}\right)=(\lambda+\mu)\mathbf{I}_2+(3\lambda+5\mu)\ln|\mathbf{x}|\mathbf{I}_2
+(3\lambda+7\mu)\frac{\mathbf{x}\mathbf{x}^{T}}{|\mathbf{x}|^2}.
$$
Therefore, we can obtain the proof directly.
\end{proof}

From Proposition \ref{fundamental solution expansion}, the single-layer potential $\mathbf{S}_{\partial D}^{\omega}$ has the following asymptotic expansion for $\omega\ll 1$,
\begin{align}\label{singlelayer}
\mathbf{S}_{\partial D}^{\omega}[\boldsymbol{\varphi}](\mathbf{x})&=\hat{\mathbf{S}}_{\partial D}^{\omega}[\boldsymbol{\varphi}](\mathbf{x})
+\omega^2\ln\omega\rho\mathbf{S}_{\partial D,1}[\boldsymbol{\varphi}](\mathbf{x})
+\omega^2\rho\ln\sqrt{\rho}\mathbf{S}_{\partial D,1}[\boldsymbol{\varphi}](\mathbf{x})\nonumber\\
&+\omega^2\rho\mathbf{S}_{\partial D,2}[\boldsymbol{\varphi}](\mathbf{x})
+\mathcal{O}(\omega^4\ln{\omega}+\omega^4),
\end{align}
where
\begin{equation}\label{hats1}
\hat{\mathbf{S}}_{\partial D}^{\omega}[\boldsymbol{\varphi}](\mathbf{x}):=\mathbf{S}_{\partial D}[\boldsymbol{\varphi}](\mathbf{x})+\gamma_{\omega}\int_{\partial D}\boldsymbol{\varphi}(\mathbf{y})d s(\mathbf{y}),
\end{equation}
$$
\mathbf{S}_{\partial D,i}[\boldsymbol{\varphi}](\mathbf{x})=\int_{\partial D} \boldsymbol{\Gamma_i}(\mathbf{x}-\mathbf{y})\boldsymbol{\varphi}(\mathbf{y})d s(\mathbf{y}),\quad i=1,2.
$$
Next, we consider the asymptotic expansion for the Neumann-Poincar\'e (N-P) operator
\begin{align}\label{nplayer}
\mathbf{K}_{\partial D}^{\omega,*}[\boldsymbol{\varphi}](\mathbf{x})&=\mathbf{K}_{\partial D}^{*}[\boldsymbol{\varphi}](\mathbf{x})
+\omega^2\ln\omega\rho\mathbf{K}_{\partial D,1}^{*}[\boldsymbol{\varphi}](\mathbf{x})
+\omega^2\rho\ln\sqrt{\rho}\mathbf{K}_{\partial D,1}^{*}[\boldsymbol{\varphi}](\mathbf{x})\nonumber\\
&+\omega^2\rho\mathbf{K}_{\partial D,2}^{*}[\boldsymbol{\varphi}](\mathbf{x})+\mathcal{O}(\omega^4\ln{\omega}+\omega^4),
\end{align}
where
$$
\mathbf{K}_{\partial D,i}^{*}[\boldsymbol{\varphi}](\mathbf{x})=\int_{\partial D} \partial_{\boldsymbol{\nu}_{\mathbf{x}}} \boldsymbol{\Gamma_i}(\mathbf{x}-\mathbf{y})\boldsymbol{\varphi}(\mathbf{y})d s(\mathbf{y}),\quad i=1,2.
$$
Let $\tilde{\mathbf{S}}_{\partial D}^{\omega}$ be the single-layer potential operator associated with the parameters $(\tilde{\lambda}, \tilde{\mu},\tilde{\rho})$. Then by using Proposition \ref{fundamental tilde solution expansion}, there exists the following asymptotic analysis.
\begin{align}\label{tildesinglelayer}
\tilde{\mathbf{S}}_{\partial D}^{\omega}[\boldsymbol{\varphi}](\mathbf{x})&=\delta\hat{\mathbf{S}}_{\partial D}^{\tau\omega}[\boldsymbol{\varphi}](\mathbf{x})
+\delta\omega^2\ln\omega\rho\tau^2\mathbf{S}_{\partial D,1}[\boldsymbol{\varphi}](\mathbf{x})+\delta\omega^2\rho\tau^2 \ln(\sqrt{\rho}\tau)\mathbf{S}_{\partial D,1}[\boldsymbol{\varphi}](\mathbf{x})\nonumber\\
&+\delta\omega^2\rho\tau^2\mathbf{S}_{\partial D,2}[\boldsymbol{\varphi}](\mathbf{x})+\mathcal{O}(\delta\tau^4\omega^4\ln{\omega}+\delta\tau^4\omega^4),%\nonumber
\end{align}
where
\begin{equation}\label{hats2}
\hat{\mathbf{S}}_{\partial D}^{\tau\omega}[\boldsymbol{\varphi}](\mathbf{x}):=\mathbf{S}_{\partial D}[\boldsymbol{\varphi}](\mathbf{x})+\gamma_{\tau\omega}\int_{\partial D}\boldsymbol{\varphi}(\mathbf{y})d s(\mathbf{y}).
\end{equation}
Moreover,
\begin{align}\label{tildenplayer}
\tilde{\mathbf{K}}_{\partial D}^{\omega,*}[\boldsymbol{\varphi}](\mathbf{x})&=\mathbf{K}_{\partial D}^{*}[\boldsymbol{\varphi}](\mathbf{x})
+\omega^2\ln\omega\rho\tau^2\mathbf{K}_{\partial D,1}^{*}[\boldsymbol{\varphi}](\mathbf{x})+\omega^2\rho\tau^2 \ln(\sqrt{\rho}\tau)\mathbf{K}_{\partial D,1}^{*}[\boldsymbol{\varphi}](\mathbf{x})\nonumber\\
&+\omega^2\rho\tau^2\mathbf{K}_{\partial D,2}^{*}[\boldsymbol{\varphi}](\mathbf{x})+\mathcal{O}(\tau^4\omega^4\ln{\omega}+\tau^4\omega^4).
\end{align}

Let $\mathbf{\Psi}$ be the vector space of all linear solutions to the equation $\mathcal{L}_{\lambda,\mu}\mathbf{u}=0$ satisfying $\partial_{\boldsymbol{\nu}}\mathbf{u}=0$ on $\partial D$, or, equivalently,
\[
\mathbf{\Psi}:=\left\{\boldsymbol{\psi}:\partial_i\psi_j+\partial_j\psi_i=0,\ 1\leq i,j\leq 2\right\},
\]
where $\psi_i\left(1\leq i\leq2\right)$ denote the components of $\boldsymbol{\psi}$. Define a subspace of $L^2(\partial D)^2$ by
$$
L_{\mathbf{\Psi}}^2(\partial D)=\left\{\boldsymbol{f} \in L^2(\partial D)^2: \int_{\partial D} \boldsymbol{f} \cdot \boldsymbol{\psi} d \sigma=0 ~\forall \boldsymbol{\psi} \in \mathbf{\Psi}\right\}.
$$
Notice that the dimension of the space $\mathbf{\Psi}$ is 3, which is spanned by \cite{Kangelasticity}
\begin{equation}\label{xi}
\boldsymbol{\xi}_1=
\left(\begin{array}{l}
1  \\
0
\end{array}\right),\quad\quad
\boldsymbol{\xi}_2=
\left(\begin{array}{l}
0 \\
1
\end{array}\right),\quad\quad
\boldsymbol{\xi}_3=
\left(\begin{array}{l}
x_{2} \\
-x_{1}
\end{array}\right).%\nonumber
\end{equation}
\begin{lem}(see \cite{HAmmariElasticityImaging})\label{Xi}   %参考Li dipolar resonaces
The kernel of the operator $-\mathbf{I}_2 / 2+\mathbf{K}_{\partial D}$ coincides with the space $\mathbf{\Psi}$, where
$\mathbf{K}_{\partial D}$ is the adjoint operator of $\mathbf{K}_{\partial D}^{*}$.
\end{lem}
Denote by $\mathbf{\Psi}^*$
the kernel of the operator $-\mathbf{I}_2 / 2+\mathbf{K}_{\partial D}^*$. The following facts can be found in \cite{Duke1988}.
\begin{lem}
The dimension of the kernel of the operator $-\mathbf{I}_2 / 2+\mathbf{K}_{\partial D}^*$ is 3, i.e. $\text{dim}(\mathbf{\Psi}^*)=3$.
\end{lem}
\begin{lem}(see \cite{ChristianConstanda})\label{kernelspace}
For any closed $C^2$ curve $\partial D$, there is a unique vector field $\mathbf{G}=(\boldsymbol{\zeta}_1,\boldsymbol{\zeta}_2,\boldsymbol{\zeta}_3)$, and a $3\times 3$ constant matrix ${\mathbf{C}}=(C_{ij})_{i,j=1}^{3}$, such that $\boldsymbol{\zeta}_i(1\leq i\leq 3)$ is a basis for the kernel space of $-\mathbf{I}_2 / 2+\mathbf{K}_{\partial D}^{*}$ and
\begin{equation}
\mathbf{S}_{\partial D}[\mathbf{G}]=\mathbf{F}\mathbf{C},\quad \int_{\partial D}\mathbf{F}(\mathbf{y})^T\mathbf{G}(\mathbf{y})d s(\mathbf{y})={\mathbf{I}_3},\nonumber
\end{equation}
where $\mathbf{F}=(\boldsymbol{\xi}_1,\boldsymbol{\xi}_2,\boldsymbol{\xi}_3)$, and $\mathbf{I}_3$ stands for the $3\times3$ identity matrix.
\end{lem}
\begin{lem}
The constant matrix ${\mathbf{C}}$ is symmetric.
\begin{proof}
For fixed $i$, $j$, denote $\boldsymbol{\zeta}_i=\left(\boldsymbol{\zeta}_{i1},\boldsymbol{\zeta}_{i2}\right)^T$, $\boldsymbol{\zeta}_j=\left(\boldsymbol{\zeta}_{j1},\boldsymbol{\zeta}_{j2}\right)^T$, and
$$
\boldsymbol{\Gamma}:=
\begin{pmatrix}
\Gamma_{11} &  \Gamma_{12} \\
\Gamma_{12} &  \Gamma_{22} \\
\end{pmatrix}.
$$
Then from Lemma \ref{kernelspace}, combined with integration by parts, we can obtain
\begin{align}
C_{ji}&=\Big(C_{1i}\boldsymbol{\xi}_1+C_{2i}\boldsymbol{\xi}_2+C_{3i}\boldsymbol{\xi}_3,\boldsymbol{\zeta}_j \Big)=\Big(\mathbf{S}_{\partial D}[\boldsymbol{\zeta}_i],\boldsymbol{\zeta}_j \Big)=\int_{\partial D}\mathbf{S}_{\partial D}[\boldsymbol{\zeta}_i]\cdot\boldsymbol{\zeta}_j\nonumber\\
&=\int_{\partial D}\int_{\partial D}\left(\Gamma_{11}\boldsymbol{\zeta}_{i1}+\Gamma_{12}\boldsymbol{\zeta}_{i2}\right)\boldsymbol{\zeta}_{j1}
+\left(\Gamma_{12}\boldsymbol{\zeta}_{i1}+\Gamma_{22}\boldsymbol{\zeta}_{i2}\right)\boldsymbol{\zeta}_{j2}\nonumber\\
&=\int_{\partial D}\int_{\partial D}\left(\Gamma_{11}\boldsymbol{\zeta}_{j1}+\Gamma_{12}\boldsymbol{\zeta}_{j2}\right)\boldsymbol{\zeta}_{i1}
+\left(\Gamma_{12}\boldsymbol{\zeta}_{j1}+\Gamma_{22}\boldsymbol{\zeta}_{j2}\right)\boldsymbol{\zeta}_{i2}=\Big(\mathbf{S}_{\partial D}[\boldsymbol{\zeta}_j],\boldsymbol{\zeta}_i \Big)=C_{ij}.\nonumber
\end{align}
This completes the proof.
\end{proof}
\end{lem}
The corresponding space for the static Lam\'e system with Neumann boundary conditions is three-dimensional, which affects the invertibility of $\hat{\mathbf{S}}_{\partial D}^{\omega}$. Note that the  invertibility of operator ${\mathbf{S}}_{\partial D}$ is equivalent to that of $\mathbf{C}$ in Lemma \ref{kernelspace}. In fact, if $\mathbf{C}$ is singular, then there exists a non-zero vector $\mathbf{z}$ such that $\mathbf{C}\mathbf{z}=0$. Thus we have
$$
\mathbf{S}_{\partial D}[\mathbf{G}\mathbf{z}]=\mathbf{S}_{\partial D}[\mathbf{G}]\mathbf{z}=\mathbf{F}(\mathbf{C}\mathbf{z})=0.
$$
Namely, the operator ${\mathbf{S}}_{\partial D}$ is not invertible. Based on this idea, we can derive the following results.

\begin{prop}\label{matrixC}
For $\omega\ll1$, if the operator $\hat{\mathbf{S}}_{\partial D}^{\omega}$ is not invertible in $\mathcal{L}(L^{2}\left(\partial D)^{2},H^{1}(\partial D)^{2}\right)$, then the matrix $\mathbf{C}$ satisfies the following conditions% \textcolor{red}{if and only if}
$$
\mathbf{C}=
\begin{pmatrix}
C_{11} &  C_{12} & 0\\
C_{12} &  C_{22} & 0\\
0 &  0 & 0
\end{pmatrix}.
$$
\end{prop}
\begin{proof}
For any function $\boldsymbol{\varphi} \in L^{2}(\partial D)^{2}$, note that
$$
\gamma_{\omega}\int_{\partial D}\boldsymbol{\varphi}=\gamma_{\omega}\boldsymbol{\xi}_1\int_{\partial D}\boldsymbol{\xi}_1\cdot\boldsymbol{\varphi}
+\gamma_{\omega}\boldsymbol{\xi}_2\int_{\partial D}\boldsymbol{\xi}_2\cdot\boldsymbol{\varphi},
$$
which further yields that
$$
\gamma_{\omega}\int_{\partial D}\boldsymbol{\zeta}_1=\gamma_{\omega}\boldsymbol{\xi}_1, \quad\quad \gamma_{\omega}\int_{\partial D}\boldsymbol{\zeta}_2=\gamma_{\omega}\boldsymbol{\xi}_2,\quad \quad\gamma_{\omega}\int_{\partial D}\boldsymbol{\zeta}_3=0.
$$
Then, by straightforward computations, one can see that
$$
\hat{\mathbf{S}}_{\partial D}^{\omega}[\mathbf{G}]=\mathbf{S}_{\partial D}[\mathbf{G}](\mathbf{x})+\gamma_{\omega}\int_{\partial D}\mathbf{G}(\mathbf{y})d s(\mathbf{y}):=\mathbf{F}\mathscr{C},
$$
where
$$
\mathscr{C}=
\begin{pmatrix}
C_{11}+\gamma_{\omega} &  C_{12} & C_{13}\\
C_{21} &  C_{22}+\gamma_{\omega} & C_{23}\\
C_{31} &  C_{32} & C_{33}
\end{pmatrix}.
$$
Since $\hat{\mathbf{S}}_{\partial D}^{\omega}$ is not invertible, then
\begin{align}
\det\mathscr{C}&=C_{33}\gamma_{\omega}^2+\Big((C_{11}+C_{22})C_{33}-C_{13}C_{31}-C_{23}C_{32}\Big)\gamma_{\omega}+C_{11}C_{22}C_{33}+C_{12}C_{23}C_{31}\nonumber\\
&+C_{21}C_{32}C_{13}-C_{22}C_{13}C_{31}-C_{11}C_{23}C_{32}-C_{12}C_{21}C_{33}=0.\nonumber
\end{align}
By the symmetry argument for $\mathbf{C}$, summing relations $\gamma_{\omega}\rightarrow\infty$ as $\omega\rightarrow0$, Therefore we have
$$
C_{13}=C_{31}=C_{23}=C_{32}=C_{33}=0.
$$
The proof is then completed.
\end{proof}
We have proved the following Theorem, which is the main result of this paper.
\begin{thm}
For $\omega\ll1$, the operator $\hat{\mathbf{S}}_{\partial D}^{\omega}$ is invertible in $\mathcal{L}(L^{2}\left(\partial D)^{2},H^{1}(\partial D)^{2}\right)$.%defined in \eqref{hats1}
\end{thm}
\begin{proof}
In view of $\int_{\partial D}\boldsymbol{\zeta}_3=0$ for $\boldsymbol{\zeta}_3\in L^{2}\left(\partial D\right)^{2}\neq0$, we can identify
\begin{equation}\label{c33}
C_{33}=\Big(\mathbf{S}_{\partial D}[\boldsymbol{\zeta}_3],\boldsymbol{\zeta}_3 \Big)<0.
\end{equation}
In fact, let the Poisson ratio $\nu=\frac{\lambda}{2(\lambda+\mu)}$, the fundamental solution $U$ of the Navier equation satisfies $U(\mathbf{x})=-\boldsymbol{\Gamma}(\mathbf{x})$. So, we can obtain \eqref{c33} by combining the embedding theorem with Proposition 1 from \cite{2DelasticInvert}. Finally, the invertibility of $\hat{\mathbf{S}}_{\partial D}^{\omega}$ can be directly proven by
Proposition \ref{matrixC}.
\end{proof}

\section{The sub-wavelength resonances}
\subsection{Single resonator}In this subsection, we first consider the case of a single resonator~(see Figure \ref{fig:single} for a schematic illustration). Let $\mathbf{u}^{i}$ be a time-harmonic incident elastic wave satisfying %the elastic equation in the entire space $\mathbb{R}^{2}$
\begin{equation}\label{Lame operator}
\mathcal{L}_{\lambda, \mu} \mathbf{u}^{i}(\mathbf{x})+\omega^{2} \rho \mathbf{u}^{i}(\mathbf{x})=0,
\end{equation}
where $\omega>0$ is the angular frequency, and the Lam\'e operator $\mathcal{L}_{\lambda, \mu}$ associated with the parameters $(\lambda, \mu)$ is defined by
\begin{equation}\label{Lame definition}
\mathcal{L}_{\lambda, \mu} \mathbf{u}^{i}:=\mu \triangle \mathbf{u}^{i}+(\lambda+\mu) \nabla \nabla \cdot \mathbf{u}^{i} .
\end{equation}
Then the total displacement field $\mathbf{u}$ described above is controlled by the following system
\begin{equation}\label{Lame system1}
\begin{cases}
\mathcal{L}_{\tilde{\lambda}, \tilde{\mu}} \mathbf{u}(\mathbf{x})+\omega^{2} \tilde{\rho} \mathbf{u}(\mathbf{x})=0, & \mathbf{x} \in D , \\ \mathcal{L}_{\lambda, \mu} \mathbf{u}(\mathbf{x})+\omega^{2} \rho \mathbf{u}(\mathbf{x})=0, & \mathbf{x} \in \mathbb{R}^{2} \backslash \bar{D}, \\ \left.\mathbf{u}(\mathbf{x})\right|_{-}=\left.\mathbf{u}(\mathbf{x})\right|_{+}, & \mathbf{x} \in \partial D, \\ \left.\partial_{\tilde{\boldsymbol{\nu}}} \mathbf{u}(\mathbf{x})\right|_{-}=\left.\partial_{\boldsymbol{\nu}} \mathbf{u}(\mathbf{x})\right|_{+}, & \mathbf{x} \in \partial D ,\\
\mathbf{u}^{s}:=\mathbf{u}-\mathbf{u}^{i} \quad \text { satisfies the radiation condition}, &
\end{cases}
\end{equation}
where the subscript $\pm$ indicate the limits from outside and inside of $D$, respectively. In \eqref{Lame system1}, the co-normal derivative $\partial_{\boldsymbol{\nu}}$ associated with the parameters $(\lambda, \mu)$ is defined by
\begin{equation}\label{co-normal derivative}
\partial_{\boldsymbol{\nu}} \mathbf{u}=\lambda(\nabla \cdot \mathbf{u}) \boldsymbol{\nu}+2 \mu\left(\nabla^{s} \mathbf{u}\right) \boldsymbol{\nu},
\end{equation}
where $\boldsymbol{\nu}$ represents the outward unit normal to $\partial D$ and the operator $\nabla^{s}$ is the symmetric gradient
\begin{equation}
\nabla^{s} \mathbf{u}:=\frac{1}{2}\left(\nabla \mathbf{u}+\nabla \mathbf{u}^{T}\right),
\end{equation}
with $\nabla \mathbf{u}$ denoting the matrix $\left(\partial_{j} u_{i}\right)_{i, j=1}^{2}$ and the superscript $t$ signifying the matrix transpose. The operators $\mathcal{L}_{\tilde{\lambda}, \tilde{\mu}}$ and $\partial_{\tilde{\boldsymbol{\nu}}}$ are defined in \eqref{Lame definition} and \eqref{co-normal derivative}, respectively, with $(\lambda, \mu)$ replaced by $(\tilde{\lambda}, \tilde{\mu})$. In \eqref{Lame system1}, the radiation condition designates the following facts as $|\mathbf{x}| \rightarrow+\infty$~(see\cite{HAmmariElasticityImaging}),
\begin{equation}
\begin{aligned}
\left(\nabla \times \nabla \times \mathbf{u}^{s}\right)(\mathbf{x}) \times \frac{\mathbf{x}}{|\mathbf{x}|}-\mathrm{i} k_{s} \nabla \times \mathbf{u}^{s}(\mathbf{x}) & =\mathcal{O}\left(|\mathbf{x}|^{-\frac{3}{2}}\right), \\
\frac{\mathbf{x}}{|\mathbf{x}|} \cdot\left[\nabla\left(\nabla \cdot \mathbf{u}^{s}\right)\right](\mathbf{x})-\mathrm{i} k_{p} \nabla \mathbf{u}^{s}(\mathbf{x}) & =\mathcal{O}\left(|\mathbf{x}|^{-\frac{3}{2}}\right).
\end{aligned}
\end{equation}
We may assume that the size of the domain $D$ is of order 1 and $k_{s}=o(1), k_{p}=o(1)$. In what follows, the parameters $\tilde{k}_{s}, \tilde{k}_{p}, \tilde{c}_{s}, \tilde{c}_{p}$ are defined in \eqref{kskp} and \eqref{cscp} by replacing $(\lambda, \mu, \rho)$ with $(\tilde{\lambda}, \tilde{\mu}, \tilde{\rho})$. From the relationship \eqref{tau2}, we further have that
\begin{equation}
\tilde{k}_{s}=\tau k_{s}=o(1), \quad \tilde{k}_{p}=\tau k_{p}=o(1) .
\end{equation}

With the help of the potential theory presented above, the solution to the system \eqref{Lame system1} can be written as
\begin{equation}\label{system solution1}
\mathbf{u}=
\begin{cases}
\tilde{\mathbf{S}}_{\partial D}^{\omega}[\boldsymbol{\varphi}](\mathbf{x}), & \mathbf{x} \in D , \\
\mathbf{S}_{\partial D}^{\omega}[\boldsymbol{\psi}](\mathbf{x})+\mathbf{u}^{i}, & \mathbf{x} \in \mathbb{R}^{2} \backslash \bar{D}.
\end{cases}
\end{equation}
for some density functions $\boldsymbol{\varphi}, \boldsymbol{\psi} \in L^{2}(\partial D)^{2}$. By matching the transmission conditions on the boundary, i.e. the third and fourth conditions in \eqref{Lame system1} and with the help of the jump formula in \eqref{jump formula}, the density functions $\boldsymbol{\varphi}$ and $\boldsymbol{\psi}$ in \eqref{system solution1} should satisfy the following system
\begin{equation}\label{Aomegadelta}
\mathcal{A}(\omega, \delta)[\Phi](\mathbf{x})=\mathcal{F}(\mathbf{x}), \quad \mathbf{x} \in \partial D,
\end{equation}
where
\begin{equation}\label{mathcalA}
\mathcal{A}(\omega, \delta)=
\left(\begin{array}{cc}
\tilde{\mathbf{S}}_{\partial D}^{\omega} & -\mathbf{S}_{\partial D}^{\omega} \\
-\frac{\mathbf{I}_2}{2}+\tilde{\mathbf{K}}_{\partial D}^{\omega, *} & -(\frac{\mathbf{I}_2}{2}+\mathbf{K}_{\partial D}^{\omega, *})
\end{array}\right), \quad
\Phi=\left(\begin{array}{c}
\boldsymbol{\varphi }\\
\boldsymbol{\psi}
\end{array}\right), \quad
\mathcal{F}=\left(\begin{array}{c}
\mathbf{u}^{i} \\
\partial_{\boldsymbol{\nu}} \mathbf{u}^{i}
\end{array}\right) .
\end{equation}
For the further discussion, we introduce the spaces {$\mathcal{H}=L^{2}(\partial D)^{2} \times L^{2}(\partial D)^{2}$ and $\mathcal{H}^{1}=$ $H^{1}(\partial D)^{2} \times L^{2}(\partial D)^{2}$}. Apparently, the operator $\mathcal{A}(\omega, \delta)$ is defined from $\mathcal{H}$ to $\mathcal{H}^{1}$.

{We first look at the limiting case when $\delta=0$ and $\omega=0$. Since $\mathbf{S}_{\partial D}^{\omega}$ exhibits a singular behavior in the limit $\omega \to 0$, to address this, we first consider the corresponding operator pencil% One can easily check that the operators $\mathcal{A}_0$ and $\mathcal{A}_0^*$ have the following expression
\begin{equation}\label{a0}
\mathcal{A}_0(\omega)=
\begin{pmatrix}
0 & -\hat{\mathbf{S}}_{\partial D}^{\omega}\\
-\frac{\mathbf{I}_2}{2}+\mathbf{K}_{\partial D}^* &
-(\frac{\mathbf{I}_2}{2}+\mathbf{K}_{\partial D}^*)\\
\end{pmatrix},\quad\quad
\mathcal{A}_0^*(\omega)=
\begin{pmatrix}
0 & -\frac{\mathbf{I}_2}{2}+\mathbf{K}_{\partial D}\\
-\hat{\mathbf{S}}_{\partial D}^{{\omega},* }&
-(\frac{\mathbf{I}_2}{2}+\mathbf{K}_{\partial D})\\
\end{pmatrix},
\end{equation}
where $\mathcal{A}_0^*(\omega)$ and $\mathbf{S}_{\partial D}^*$ are the adjoint of $\mathcal{A}_0(\omega)$ and $\mathbf{S}_{\partial D}$,  together with
$$
\hat{\mathbf{S}}_{\partial D}^{{\omega},* }[\boldsymbol{\varphi}](\mathbf{x}):=\mathbf{S}_{\partial D}^{*}[\boldsymbol{\varphi}](\mathbf{x})+\gamma_{\omega}\int_{\partial D}\boldsymbol{\varphi}(\mathbf{y})d s(\mathbf{y}).
$$
To simplify notation, we shall use $\mathcal{A}_0$ to denote $\mathcal{A}_0({\omega})$ in what follows. The analysis now focuses on the kernel of $\mathcal{A}_0$ and $\mathcal{A}^*_0$, one can easily check the following Lemma holds.}
\begin{lem}\label{ker space2}
We have
\begin{itemize}
\item[(i)] $\operatorname{ker} \mathcal{A}_0=\operatorname{span}\{\hat{\Phi}_{i}\}(1\leq i\leq 3)$, where
$$
\hat{\Phi}_{i}=
\begin{pmatrix}
\boldsymbol{\zeta}_i\\
0
\end{pmatrix}.
$$
Moreover, the functions satisfy $(\hat{\Phi}_{i},\hat{\Phi}_{j})={\delta_{ij}}$.
\item[(ii)] $\operatorname{ker} \mathcal{A}^*_0=\operatorname{span}\{\hat{\Psi}_{i}\}(1\leq i\leq 3)$, where
$$
\hat{\Psi}_{i}=d_{i}
\begin{pmatrix}
{-(\hat{\mathbf{S}}_{\partial D}^{{\omega},* })^{-1}[\boldsymbol{\xi}_i]}\\
\boldsymbol{\xi}_i
\end{pmatrix}.
$$
The selection of constant $d_{i}$ is such that $(\hat{\Psi}_{i},\hat{\Psi}_{j})={\delta_{ij}}$.
\end{itemize}
\end{lem}
\begin{proof}
This proof can be obtained by combining Lemmas \ref{Xi} and \ref{kernelspace}, as well as expression \eqref{a0}.
\end{proof}

\begin{rem}
Here, we need to emphasize two points. On the one hand, the basis function $\hat{\Psi}_{i}(1\leq i\leq 3)$ of kernel space $\mathcal{A}^*_0$ is related to $\omega$, which is different from the situation in three-dimensions. On the other hand, the functions $\boldsymbol{\xi}_i(1\leq i\leq 3)$ and $\boldsymbol{\zeta}_i(1\leq i\leq 3)$ defined in Lemma \ref{ker space2} may not be exactly same as these in equation \eqref{xi} and Lemma \ref{kernelspace}. {In fact}, by using translation transformation~(i.e. moving the center of gravity to the origin), coefficient reconfiguration, and the diagonalisation of symmetric matrices, we can obtain that $(\boldsymbol{\zeta}_i,\boldsymbol{\xi}_j)=\delta_{ij}$, $(\boldsymbol{\zeta}_i,\boldsymbol{\zeta}_j)=\delta_{ij}$ and $(\boldsymbol{\xi}_i,\boldsymbol{\xi}_j)=\lambda_i\delta_{ij}$ for fixed $i$ and $j$.
\end{rem}

According to the Gohberg-Sigal theory \cite{SpectralAnalysis}, the following result regarding the existence of sub-wavelength resonances in system \eqref{Lame system1} can be obtained.
\begin{lem}
For any $\delta$ sufficiently small, if the parameter $\epsilon\ll 1$, then there exists characteristic value $\omega^* \ll 1$ for the operator-valued analytic function $\mathcal{A}(\omega,\delta)$. Moreover, $\omega^*$ depends on $\delta$ continuously.
\end{lem}

\begin{lem}
The operator $\mathcal{A}(\omega, \delta)$ defined in \eqref{mathcalA} has the following asymptotic expansion
$$
\mathcal{A}(\omega,\delta)=\mathcal{A}_0+\mathcal{B}(\omega,\delta):=\mathcal{A}_0+\omega^2\ln\omega\mathcal{A}_{1,0}+\omega^2\mathcal{A}_{2,0}+\delta\mathcal{A}_{0,1}
+\mathcal{O}(\delta\omega^2\ln{\omega}+\omega^4\ln{\omega}),
$$
where
$$
\mathcal{A}_{1,0}=
\begin{pmatrix}
0 & -\rho{\mathbf{S}}_{\partial D,1}\\
\rho\tau^2\mathbf{K}_{\partial D,1}^{*} &
-\rho\mathbf{K}_{\partial D,1}^{*}\\
\end{pmatrix},\quad
\mathcal{A}_{0,1}=
\begin{pmatrix}
\hat{\mathbf{S}}_{\partial D}^{\tau\omega} & 0\\
0 & 0\\
\end{pmatrix},
$$
and
$$
\mathcal{A}_{2,0}=
\begin{pmatrix}
0 & -\rho\ln\sqrt{\rho}\mathbf{S}_{\partial D,1}-\rho\mathbf{S}_{\partial D,2}\\
\rho\tau^2\ln(\sqrt{\rho}\tau)\mathbf{K}_{\partial D,1}^{*} +\rho\tau^2\mathbf{K}_{\partial D,2}^{*} &
-\rho\ln\sqrt{\rho}\mathbf{K}_{\partial D,1}^{*}-\rho\mathbf{K}_{\partial D,2}^{*}\\
\end{pmatrix}.\quad
$$
\end{lem}
\begin{proof}
The proof directly follows from asymptotic expansions defined in \eqref{singlelayer} \eqref{nplayer} as well as \eqref{tildesinglelayer} and \eqref{tildenplayer}.
\end{proof}

\begin{rem}\label{regime_rem}
		It is noted that we assume $\delta \ln \omega \ll 1$ in the asymptotic expansion of $\mathcal{A}(\omega,\delta)$ above and so that $\Vert \mathcal{A}_{0,1} \Vert \ll 1$. $\mathcal{A}(\omega,\delta)$ can be regarded as $\mathcal{A}_0$ together with a small perturbation $\mathcal{B}(\omega,\delta)$. Without this assumption, the operator $\mathcal{A}_({\omega,\delta})$ would be invertible, and no characteristic values would lie within this range.
\end{rem}

We now perturb $\mathcal{A}_0$ by a operator $\mathcal{P}$ given by
$$
\mathcal{P}_{j}[\Phi]=(\Phi,\hat{\Phi}_{j})\hat{\Psi}_{j},\quad 1 \leq j \leq 3,
$$
and denote it by
\begin{equation}\label{hatA0}
\hat{\mathcal{A}_0}:=\mathcal{A}_0+\mathcal{P}=\mathcal{A}_{0}+\sum_{j=1}^{3} \mathcal{P}_{j}.
\end{equation}
Then the operator $\hat{\mathcal{A}}_{0}$ shares the following properties.
\begin{lem}
There holds
\begin{itemize}
\item[(i)] the operator $\hat{\mathcal{A}}_{0}$ is {bijective} in $\mathcal{L}\left(\mathcal{H}, \mathcal{H}^{1}\right)$ and $\hat{\mathcal{A}}_{0}[\hat{\Phi}_{i}]=\hat{\Psi}_{i}(1 \leq i \leq 3)$.
\item[(ii)]the adjoint operator of $\hat{\mathcal{A}}_{0}$, i.e. $\hat{\mathcal{A}}_{0}^{*}$, is also {bijective} and $\hat{\mathcal{A}}_{0}^{*}[\hat{\Psi}_{i}]=\hat{\Phi}_{i}(1 \leq i \leq 3)$.
\end{itemize}
\end{lem}
\begin{proof}
Note that the operator ${\mathcal{A}}_{0}$ is a Fredholm operator with the index 0. Thus the {bijection} of the operator $\hat{\mathcal{A}}_{0}$ directly follows from that the construction of the operator
$\hat{\mathcal{A}}_{0}$ in \eqref{hatA0}. We can confirm that
\begin{equation}
\hat{\mathcal{A}}_{0}[\hat{\Phi}_{i}]=(\mathcal{A}_0+\mathcal{P})[\hat{\Phi}_{i}]=\sum_{j=1}^{3}(\hat{\Phi}_{i},\hat{\Phi}_{j})\hat{\Psi}_{j}
={\delta_{ij}}\hat{\Psi}_{i}=\hat{\Psi}_{i}.\nonumber
\end{equation}
And in a similar manner,
\begin{equation}
\hat{\mathcal{A}}_{0}^{*}[\hat{\Psi}_{i}]=(\mathcal{A}_0+\mathcal{P})^{*}[\hat{\Psi}_{i}]=\sum_{j=1}^{3}(\hat{\Psi}_{i},\hat{\Psi}_{j})\hat{\Phi}_{j}
={\delta_{ij}}\hat{\Phi}_{i}=\hat{\Phi}_{i}.\nonumber
\end{equation}
This completes the proof.
\end{proof}
In order to facilitate the presentation of our ideas in this part, we introduce the $3\times3$ matrices $\mathbf{P}$, $\mathbf{M}$ and $\mathbf{Q}$ as follows,
\begin{equation}
\mathbf{P}=(P_{ij})_{i,j=1}^{3},\quad\quad \mathbf{M}=(M_{ij})_{i,j=1}^{3},\quad\quad \mathbf{Q}=(Q_{ij})_{i,j=1}^{3}, \nonumber
\end{equation}
with
\begin{equation}\label{matrixQ1}
P_{ij}=\Big(\boldsymbol{\zeta}_i,\mathbf{K}_{\partial D,1}[\boldsymbol{\xi}_j]\Big),\quad\quad
M_{ij}=\Big(\boldsymbol{\zeta}_i,\mathbf{K}_{\partial D,2}[\boldsymbol{\xi}_j]\Big),\quad\quad
Q_{ij}=\Big(\hat{\mathbf{S}}_{\partial D}^{{\tau\omega} }[\boldsymbol{\zeta}_i],(\hat{\mathbf{S}}_{\partial D}^{{\omega},* })^{-1}[\boldsymbol{\xi}_j]\Big).
\end{equation}
Then, we can obtain the following Proposition.
\begin{prop}\label{matrices p and m}
The elements of the matrices $\mathbf{P}$ and $\mathbf{M}$ are given by
$$
P_{ij}=a_{\lambda,\mu}\int_{ \partial D}\boldsymbol{\zeta}_i(\mathbf{x})ds(\mathbf{x})\cdot \int_{ D}\boldsymbol{\xi}_j(\mathbf{x})d\mathbf{x},\quad\\
$$
$$
M_{ij}={b}_{\lambda,\mu}\int_{ \partial D}\boldsymbol{\zeta}_i(\mathbf{x})ds(\mathbf{x}) \cdot\int_{ D}\boldsymbol{\xi}_j(\mathbf{x})d\mathbf{x}
-\int_{ D} {\mathbf{S}}_{\partial D}[\boldsymbol{\zeta}_i](\mathbf{x}) \cdot \boldsymbol{\xi}_j(\mathbf{x}) d\mathbf{x},
$$
where the constants ${a}_{\lambda,\mu}$ and ${b}_{\lambda,\mu}$ are given by \eqref{lamea} and \eqref{lameb}, respectively.
\end{prop}
\begin{proof}
By using Green's formula and Lemma \ref{calculatelame}, there holds
\begin{align}
\mathbf{K}_{\partial D,1}[\boldsymbol{\xi}_j](\mathbf{x})&=\int_{\partial D}\partial_{\boldsymbol{\nu}_{\mathbf{y}}} \boldsymbol{\Gamma_1}(\mathbf{x}-\mathbf{y})\boldsymbol{\xi}_j(\mathbf{y})d s(\mathbf{y})=\int_{ D} \mathcal{L}_{\lambda,\mu} \boldsymbol{\Gamma_1}(\mathbf{x}-\mathbf{y})\boldsymbol{\xi}_j(\mathbf{y})d\mathbf{y}\nonumber\\
&=a_{\lambda,\mu}\int_{ D}\boldsymbol{\xi}_j(\mathbf{y})d\mathbf{y},\nonumber
\end{align}
and
\begin{align}
&\mathbf{K}_{\partial D,2}[\boldsymbol{\xi}_j](\mathbf{x})=\int_{\partial D}\partial_{\boldsymbol{\nu}_{\mathbf{y}}} \boldsymbol{\Gamma_2}(\mathbf{x}-\mathbf{y})\boldsymbol{\xi}_j(\mathbf{y})d s(\mathbf{y})=\int_{ D} \mathcal{L}_{\lambda,\mu} \boldsymbol{\Gamma_2}(\mathbf{x}-\mathbf{y})\boldsymbol{\xi}_j(\mathbf{y})d\mathbf{y}\nonumber\\
&=b_{\lambda,\mu}\int_{ D}\boldsymbol{\xi}_j(\mathbf{y})d\mathbf{y}-\int_{D}\boldsymbol{\Gamma}(\mathbf{x}-\mathbf{y})\boldsymbol{\xi}_j(\mathbf{y})d\mathbf{y}.\nonumber
\end{align}
Therefore, we can get the proof from definition \eqref{matrixQ1} easily.
\end{proof}
The existence of sub-wavelength resonance frequencies is stated in the following theorem.
\begin{thm}\label{thmomega}
For $\delta\ll 1$, and $\epsilon\ll 1$, there exists three sub-wavelength resonance frequencies, counted with their multiplicities. The leading order terms are given by the roots of the following equation
\begin{equation}\label{formula1}
\det\Big(  \rho\omega^2\ln\omega\mathbf{P}+\rho\omega^2\left(\ln{(\sqrt{\rho}\tau)}\mathbf{P}
+\mathbf{M}\right)-\epsilon\mathbf{Q}\Big)=0,
\end{equation}
where the matrices $\mathbf{P}$, $\mathbf{M}$ and $\mathbf{Q}$ are given in Proposition \ref{matrices p and m} and equation \eqref{matrixQ1}, respectively. In fact, since $\mathbf{Q}$ is invertible when $\omega$ is small enough, equation \eqref{formula1} is third-order with respect to $\epsilon$, meaning that it has three roots $\omega_{i}(\epsilon)(1\leq i\leq3)$.
\end{thm}
\begin{proof}
We would like to find the characteristic values $\omega^{*} \ll 1$ such that there exits a nontrivial function $\Phi_{\delta}$ satisfying
\begin{equation}\label{Phidelta}
\mathcal{A}(\omega^{*}, \delta)[\Phi_{\delta}]=0.
\end{equation}
Lemma \ref{ker space2} shows that
$$
\mathcal{A}(0,0)[\Phi_{0}]=0,
$$ where $\Phi_{0}=\sum_{j=1}^{3} c_{j} \hat{\Phi}_{j}$ and the coefficients $c_{j}$ are arbitrary. Thus one can treat $\Phi_{\delta}$ as a perturbation of $\Phi_{0}$. We express the function $\Phi_{\delta}$ by
$$
\Phi_{\delta}=\Phi_{0}+\Phi_{1} \quad \text{with}\quad (\Phi_{1}, \Phi_{0})_{\mathcal{H}}=0.
$$
Using the definition of operator $\hat{\mathcal{A}}_{0}$ given in \eqref{hatA0}, the equation \eqref{Phidelta} is equivalent to
$$
(\hat{\mathcal{A}}_{0}+\mathcal{B}-\mathcal{P})[\Phi_{0}+\Phi_{1}]=0.
$$
%this yields that
%$$
%\Phi_{1}=(\hat{\mathcal{A}}_{0}+\mathcal{B})^{-1} \mathcal{P}\left[\Phi_{0}\right]-\Phi_{0}.
%$$
%Since the operator $\hat{\mathcal{A}}_{0}+\mathcal{B}$ is invertible for sufficient small $\omega$ and $\delta$, the operator $(\hat{\mathcal{A}}_{0}+\mathcal{B})^{-1}$ has the following asymptotic expansion
%\begin{align}
%&(\hat{\mathcal{A}}_{0}+\mathcal{B})^{-1}=(\mathbf{I}_2+\hat{\mathcal{A}}_{0}^{-1} \mathcal{B})^{-1} \hat{\mathcal{A}}_{0}^{-1} \nonumber\\
%&=\left(\mathbf{I}_2-\hat{\mathcal{A}}_{0}^{-1} \mathcal{B}+(\hat{\mathcal{A}}_{0}^{-1} \mathcal{B})^{2}+\cdots\right) \hat{\mathcal{A}}_{0}^{-1}.\nonumber
%\end{align}
By using the following fact $\mathcal{P}[\Phi_{0}]=\Psi_{0}:=\sum_{j=1}^{3} c_{j} \hat{\Psi}_{j}$, we have
\begin{align}\label{0}
0=(\Phi_{1},\Phi_{0})_{\mathcal{H}}&=\left((\hat{\mathcal{A}}_{0}+\mathcal{B})^{-1}\mathcal{P}[\Phi_{0}]-\Phi_{0},\Phi_{0}\right)_{\mathcal{H}} \nonumber\\
&=\left((\mathbf{I}_2-\hat{\mathcal{A}}_{0}^{-1}\mathcal{B}+(\hat{\mathcal{A}}_{0}^{-1} \mathcal{B})^{2}+\cdots){\Phi_{0}},{\Phi_{0}}\right)_{\mathcal{H}}-\Big(\Phi_{0}, \Phi_{0}\Big)_{\mathcal{H}} \nonumber \\
&=\left((-\hat{\mathcal{A}}_{0}^{-1}\mathcal{B}+(\hat{\mathcal{A}}_{0}^{-1}\mathcal{B})^{2}+\cdots)\Phi_{0},\Phi_{0}\right)_{\mathcal{H}}\nonumber\\
&=\left(-\hat{\mathcal{A}}_{0}^{-1}\mathcal{B}[\Phi_{0}],\Phi_{0}\right)+\mathcal{O}\left(\delta\omega^2\ln{\omega}+\omega^4\ln{\omega}\right),
\end{align}
where $\mathcal{B}=\omega^2\ln\omega\mathcal{A}_{1,0}+\omega^2\mathcal{A}_{2,0}+\delta\mathcal{A}_{0,1}
+\mathcal{O}(\delta\omega^2\ln{\omega}+\omega^4\ln{\omega})$. Moreover, let $\mathbf{\Lambda}$ be a diagonal matrix with $\Lambda_{ii}=d_i(1 \leq i \leq 3)$ and $\mathbf{c}=(c_1,c_2,c_3)^T$. Then, we have
\begin{itemize}
\item Calculation of $(\mathcal{A}_{1,0}[\Phi_0],\Psi_0)$.
\end{itemize}
\begin{align}\label{A10}
&(\mathcal{A}_{1,0}[\Phi_0],\Psi_0)=\sum_{i,j=1}^{3}c_{i}c_{j}d_{j}\rho\tau^2\Big(\boldsymbol{\zeta}_i,\mathbf{K}_{\partial D,1}[\boldsymbol{\xi}_j]\Big)
=\sum_{i,j=1}^{3} c_{i}c_{j}d_{j}\rho\tau^2 P_{ij}=\rho\tau^2\mathbf{c}^{T}\mathbf{P}\mathbf{\Lambda}\mathbf{c},%\nonumber
\end{align}
\begin{itemize}
\item Calculation of $(\mathcal{A}_{2,0}[\Phi_0],\Psi_0)$.
\end{itemize}
\begin{align}\label{A20}
&(\mathcal{A}_{2,0}[\Phi_0],\Psi_0)=\sum_{i,j=1}^{3}c_{i}c_{j}d_{j}\rho\tau^2 \Big[\ln{(\sqrt{\rho}\tau)}\Big(\boldsymbol{\zeta}_i,\mathbf{K}_{\partial D,1}[\boldsymbol{\xi}_j]\Big)+\Big(\boldsymbol{\zeta}_i,\mathbf{K}_{\partial D,2}[\boldsymbol{\xi}_j]\Big)\Big]
\nonumber\\
&=\sum_{i,j=1}^{3}c_{i}c_{j}d_{j}\rho\tau^2  \Big(\ln{(\sqrt{\rho}\tau)}P_{ij}+M_{ij} \Big)=\rho\tau^2\Big(\ln{(\sqrt{\rho}\tau)}\mathbf{c}^{T}\mathbf{P}\mathbf{\Lambda}\mathbf{c}+\mathbf{c}^{T}\mathbf{M}\mathbf{\Lambda}\mathbf{c}\Big).
\end{align}
\begin{itemize}
\item Calculation of $(\mathcal{A}_{0,1}[\Phi_0],\Psi_0)$.
\end{itemize}
\begin{align}\label{A01}
&(\mathcal{A}_{0,1}[\Phi_0],\Psi_0)=-\sum_{i,j=1}^{3}c_{i}c_{j}d_{j}\left(\hat{\mathbf{S}}_{\partial D}^{{\tau\omega} }[\boldsymbol{\zeta}_i],(\hat{\mathbf{S}}_{\partial D}^{{\omega},* })^{-1}[\boldsymbol{\xi}_j]\right)=-\sum_{i,j=1}^{3}c_{i}c_{j}d_{j}Q_{ij}=-\mathbf{c}^{T}\mathbf{Q}\mathbf{\Lambda}\mathbf{c}.
\end{align}
Therefore, combining \eqref{0}, \eqref{A10}, \eqref{A20}, and \eqref{A01}, there holds
\begin{align}
&\mathbf{c}^{T}\Big[  \rho\tau^2\omega^2\ln\omega\mathbf{P}+\rho\tau^2\omega^2\Big(\ln{(\sqrt{\rho}\tau)}\mathbf{P}
+\mathbf{M}\Big)-\delta\mathbf{Q}\Big] \mathbf{\Lambda}\mathbf{c}+\mathcal{O}(\delta\omega^2\ln{\omega}+\omega^4\ln{\omega})=0.\nonumber
\end{align}
The proof is completed.%using the fact that $\mathbf{Q}$ is invertible when $\omega$ is small enough. % 见稿6
\end{proof}

\subsection{$N$-nested resonators}In this subsection, we analyze the sub-wavelength resonances in multi-layer high-contrast elastic metamaterials~(i.e. $N$-nested resonators)~described by system \eqref{Lame system2}. First, the solution to system \eqref{Lame system2} can be written as
\begin{equation}\label{system solution2}
\mathbf{u}=
\begin{cases}
\mathbf{S}_{\Gamma_{1}^{+}}^{\omega}[\boldsymbol{\psi}_1^+](\mathbf{x})+\mathbf{u}^{i}, & \mathbf{x} \in \hat{D}_{0},\\
\tilde{\mathbf{S}}_{\Gamma_{j}^{+}}^{\omega}[\boldsymbol{\varphi}_j^+](\mathbf{x})+\tilde{\mathbf{S}}_{\Gamma_{j}^{-}}^{\omega}[\boldsymbol{\psi}_j^-](\mathbf{x}), & \mathbf{x} \in {D}_{j},1\leq j\leq N,\\
\mathbf{S}_{\Gamma_{j}^{-}}^{\omega}[\boldsymbol{\varphi}_j^-](\mathbf{x})+\mathbf{S}_{\Gamma_{j+1}^{+}}^{\omega}[\boldsymbol{\psi}_{j+1}^{+}](\mathbf{x}),& \mathbf{x} \in \hat{D}_{j},1\leq j\leq N-1,\\
\mathbf{S}_{\Gamma_{N}^{-}}^{\omega}[\boldsymbol{\varphi}_N^-](\mathbf{x}), & \mathbf{x} \in \hat{D}_{N},
\end{cases}
\end{equation}
where the density functions satisfy the following boundary integral equations,
\begin{equation}
\mathcal{A}(\omega, \delta)[\Phi]=\mathcal{F},\nonumber
\end{equation}
on $\mathcal{H}:=L^2(\Gamma_1^{+})^2\times L^2(\Gamma_1^{-})^2\times\cdots\times L^2(\Gamma_N^{+})^2\times L^2(\Gamma_N^{-})^2$. The $4N$-by-$4N$ matrix type operator $\mathcal{A}(\omega,\delta)$ has the triple diagonal block form
\begin{equation}\label{block diagonal2}% 矩阵
\begin{aligned}
\mathcal{A}(\omega,\delta):&=\text{diag}\left(\mathbf{T},\mathbf{M},\mathbf{N}\right)=
\begin{pmatrix}
\mathbf{M}_{{\Gamma_1^{+}}} &\mathbf{N}_{\Gamma_1^{+},\Gamma_1^{-}} & 0& \cdots & 0&0\\%0\nonumber\\
 \mathbf{T}_{\Gamma_1^{-},\Gamma_1^{+}}& \mathbf{M}_{{\Gamma_1^{-}}}&  \mathbf{N}_{\Gamma_1^{-},\Gamma_2^{+}}&  \cdots & 0 & 0\\%0\nonumber\\
 0& \mathbf{T}_{{\Gamma_2^{+}},{\Gamma_1^{-}}}&\mathbf{M}_{{\Gamma_2^{+}}}&
 \cdots&0&0\\
\vdots & \vdots & \vdots & \ddots & \vdots & \vdots \\%\nonumber\\
 0&0&0& \cdots &\mathbf{M}_{{\Gamma_N^{+}}} & \mathbf{N}_{{\Gamma_N^{+}},{\Gamma_N^{-}}} \\
  0&0&0& \cdots &\mathbf{T}_{{\Gamma_N^{-}},{\Gamma_N^{+}}}& \mathbf{M}_{{\Gamma_N^{-}}}
\end{pmatrix},%\nonumber
\end{aligned}
\end{equation}
and
$$
\Phi=\left(\boldsymbol{\psi}_{1}^+,\boldsymbol{\varphi}_{1}^+,\boldsymbol{\psi}_{1}^-,\boldsymbol{\varphi}_{1}^-,
\boldsymbol{\psi}_{2}^+,\boldsymbol{\varphi}_{2}^+,\boldsymbol{\psi}_{2}^-,\boldsymbol{\varphi}_{2}^-,
\cdots,\boldsymbol{\psi}_{N}^+,\boldsymbol{\varphi}_{N}^+,\boldsymbol{\psi}_{N}^-,\boldsymbol{\varphi}_{N}^-,\right)^T,\quad \mathcal{F}=\left(\mathbf{u}^{i},\partial_{\boldsymbol{\nu}} \mathbf{u}^{i},0,\cdots,0,0\right)^T,
$$
where
\begin{equation}
\mathbf{M}_{{\Gamma_j^+}}=
\begin{pmatrix}
-\mathbf{S}_{\Gamma_{j}^+}^{\omega} &\tilde{\mathbf{S}}_{\Gamma_{j}^+}^{\omega}  \\
-(\frac{\mathbf{I}}{2}+\mathbf{K}_{\Gamma_{j}^+}^{\omega,*}) &-\frac{\mathbf{I}}{2}+\tilde{\mathbf{K}}_{\Gamma_{j}^+}^{\omega,*}
\end{pmatrix},\quad
\mathbf{M}_{{\Gamma_j^-}}=
\begin{pmatrix}
-\tilde{\mathbf{S}}_{\Gamma_{j}^-}^{\omega} &{\mathbf{S}}_{\Gamma_{j}^-}^{\omega}  \\
-(\frac{\mathbf{I}}{2}+\tilde{\mathbf{K}}_{\Gamma_{j}^-}^{\omega,*}) &-\frac{\mathbf{I}}{2}+{\mathbf{K}}_{\Gamma_{j}^-}^{\omega,*}
\end{pmatrix},\,1\leq j\leq N,\nonumber
\end{equation}
and
\begin{equation}
\mathbf{N}_{{\Gamma_{j}^+},{\Gamma_{j}^-}}=
\begin{pmatrix}
\tilde{\mathbf{S}}_{\Gamma_{j}^+,\Gamma_{j}^-}^{\omega} &0 \\
\tilde{\mathbf{K}}_{\Gamma_{j}^+,\Gamma_{j}^-}^{\omega,*}&0
\end{pmatrix},\quad \quad
\mathbf{N}_{{\Gamma_{j}^-},{\Gamma_{j+1}^+}}=
\begin{pmatrix}
{\mathbf{S}}_{\Gamma_{j}^-,\Gamma_{j+1}^+}^{\omega} &0 \\
{\mathbf{K}}_{\Gamma_{j}^-,\Gamma_{j+1}^+}^{\omega,*}&0
\end{pmatrix},\,1\leq j\leq N,\nonumber
\end{equation}

\begin{equation}
\mathbf{T}_{{\Gamma_{j}^-},{\Gamma_{j}^+}}=
\begin{pmatrix}
0& -\tilde{\mathbf{S}}_{\Gamma_{j}^-,\Gamma_{j}^+}^{\omega}  \\
0& -\tilde{\mathbf{K}}_{\Gamma_{j}^-,\Gamma_{j}^+}^{\omega,*}
\end{pmatrix},\quad \quad
\mathbf{T}_{{\Gamma_{j}^+},{\Gamma_{j-1}^-}}=
\begin{pmatrix}
0& -{\mathbf{S}}_{\Gamma_{j}^+,\Gamma_{j-1}^-}^{\omega}  \\
0& -{\mathbf{K}}_{\Gamma_{j}^+,\Gamma_{j-1}^-}^{\omega,*}
\end{pmatrix},\,1\leq j\leq N.\nonumber
\end{equation}
In this part, we define
\begin{equation}
\mathbf{S}_{\Gamma_{k},\Gamma_{i}}^{\omega}[\boldsymbol{\varphi}](\mathbf{x}):=\int_{\Gamma_{i}}\boldsymbol{\Gamma}^{\omega}(\mathbf{x}-\mathbf{y}) \boldsymbol{\varphi}(\mathbf{y}) d s(\mathbf{y}), \quad \mathbf{x} \in \Gamma_{k}.\nonumber
\end{equation}
In addition,
\begin{equation}
\mathbf{K}_{\Gamma_{k},\Gamma_{i}}^{\omega,*}[\boldsymbol{\varphi}](\mathbf{x}):=\int_{\Gamma_{i}} \partial_{\boldsymbol{\nu}_{\mathbf{x}}} \boldsymbol{\Gamma}^{\omega}(\mathbf{x}-\mathbf{y}) \boldsymbol{\varphi}(\mathbf{y}) d s(\mathbf{y}),\quad \mathbf{x} \in \Gamma_{k}.\nonumber
\end{equation}
Set
\begin{equation}\label{single layer potential}
\mathbf{S}_{\Gamma_{i}}^{\omega}:=\mathbf{S}_{\Gamma_{i},\Gamma_{i}}^{\omega},\quad \quad \mathbf{K}_{\Gamma_{i}}^{\omega,*}:=\mathbf{K}_{\Gamma_{i},\Gamma_{i}}^{\omega,*}.\nonumber
\end{equation}

Next, we aim to solve the kernel space of the operator $\mathcal{A}(\omega,\delta)$  under the limit state $\delta=\omega=0$, i.e.  the kernel space of $\mathcal{A}_0:=\mathcal{A}(0,0)$. Specifically,
\begin{equation}\label{TMN}
\mathcal{A}_0=\text{diag}\left(\mathcal{T},\mathcal{M},\mathcal{N}\right),\quad
\mathcal{A}_0^*=\text{diag}\left(\mathcal{T}^*,\mathcal{M}^*,\mathcal{N}^*\right),
\end{equation}
where
\begin{equation}
\mathcal{M}_{{\Gamma_j^+}}=
\begin{pmatrix}
-\hat{\mathbf{S}}_{\Gamma_{j}^+}^{\omega} &0  \\
-(\frac{\mathbf{I}}{2}+\mathbf{K}_{\Gamma_{j}^+}^{*}) &-\frac{\mathbf{I}}{2}+{\mathbf{K}}_{\Gamma_{j}^+}^{*}
\end{pmatrix},\quad
\mathcal{M}_{{\Gamma_j^-}}=
\begin{pmatrix}
0 &\hat{{\mathbf{S}}}_{\Gamma_{j}^-}^{\omega}  \\
-(\frac{\mathbf{I}}{2}+{\mathbf{K}}_{\Gamma_{j}^-}^{*}) &-\frac{\mathbf{I}}{2}+{\mathbf{K}}_{\Gamma_{j}^-}^{*}
\end{pmatrix},\,1\leq j\leq N,\nonumber
\end{equation}
\begin{equation}
\mathcal{M}^{*}_{{\Gamma_j^+}}=
\begin{pmatrix}
-\hat{\mathbf{S}}_{\Gamma_{j}^+}^{\omega,*} &-(\frac{\mathbf{I}}{2}+\mathbf{K}_{\Gamma_{j}^+})  \\
0 &-\frac{\mathbf{I}}{2}+{\mathbf{K}}_{\Gamma_{j}^+}
\end{pmatrix},\quad
\mathcal{M}^{*}_{{\Gamma_j^-}}=
\begin{pmatrix}
0 & -(\frac{\mathbf{I}}{2}+{\mathbf{K}}_{\Gamma_{j}^-})  \\
\hat{{\mathbf{S}}}_{\Gamma_{j}^-}^{\omega,*} &-\frac{\mathbf{I}}{2}+{\mathbf{K}}_{\Gamma_{j}^-}
\end{pmatrix},\,1\leq j\leq N,\nonumber
\end{equation}
and
\begin{equation}
\mathcal{N}_{{\Gamma_{j}^+},{\Gamma_{j}^-}}=
\begin{pmatrix}
0 &0 \\
{\mathbf{K}}_{\Gamma_{j}^+,\Gamma_{j}^-}^{*}&0
\end{pmatrix},\quad \quad
\mathcal{N}_{{\Gamma_{j}^-},{\Gamma_{j+1}^+}}=
\begin{pmatrix}
\hat{\mathbf{S}}_{\Gamma_{j}^-,\Gamma_{j+1}^+}^{\omega} &0 \\
{\mathbf{K}}_{\Gamma_{j}^-,\Gamma_{j+1}^+}^{*}&0
\end{pmatrix},\,1\leq j\leq N,\nonumber
\end{equation}
\begin{equation}
\mathcal{N}^{*}_{{\Gamma_{j}^-},{\Gamma_{j}^+}}=
\begin{pmatrix}
0 &{\mathbf{K}}_{\Gamma_{j}^-,\Gamma_{j}^+} \\
0&0
\end{pmatrix},\quad \quad
\mathcal{N}^{*}_{\Gamma_{j+1}^+,{\Gamma_{j}^-}}=
\begin{pmatrix}
\hat{\mathbf{S}}_{\Gamma_{j+1}^+,\Gamma_{j}^-}^{\omega,*} &{\mathbf{K}}_{\Gamma_{j+1}^+,\Gamma_{j}^-} \\
0&0
\end{pmatrix},\,1\leq j\leq N,\nonumber
\end{equation}
and
\begin{equation}
\mathcal{T}_{{\Gamma_{j}^-},{\Gamma_{j}^+}}=
\begin{pmatrix}
0& 0 \\
0& -{\mathbf{K}}_{\Gamma_{j}^-,\Gamma_{j}^+}^{*}
\end{pmatrix},\quad \quad
\mathcal{T}_{{\Gamma_{j}^+},{\Gamma_{j-1}^-}}=
\begin{pmatrix}
0& -\hat{\mathbf{S}}_{\Gamma_{j}^+,\Gamma_{j-1}^-}^{\omega}  \\
0& -{\mathbf{K}}_{\Gamma_{j}^+,\Gamma_{j-1}^-}^{*}
\end{pmatrix},\,1\leq j\leq N.\nonumber
\end{equation}
\begin{equation}
\mathcal{T}^{*}_{{\Gamma_{j}^+},{\Gamma_{j}^-}}=
\begin{pmatrix}
0& 0 \\
0& -{\mathbf{K}}_{\Gamma_{j}^+,\Gamma_{j}^-}
\end{pmatrix},\quad \quad
\mathcal{T}^{*}_{{\Gamma_{j-1}^-},{\Gamma_{j}^+}}=
\begin{pmatrix}
0&  0 \\
-\hat{\mathbf{S}}_{\Gamma_{j-1}^-,\Gamma_{j}^+}^{\omega,*}& -{\mathbf{K}}_{\Gamma_{j-1}^-,\Gamma_{j}^+}
\end{pmatrix},\,1\leq j\leq N.\nonumber
\end{equation}

We can check that the operators $\mathcal{A}_0$ and $\mathcal{A}^{*}_0$, as defined in \eqref{TMN}, have the following kernel spaces, which is an extension of Lemma \ref{ker space2}.
\begin{lem}\label{ker space}
We have
$$
\operatorname{ker} \mathcal{A}_0=\operatorname{span}
\left\{\boldsymbol{\phi}_{1,j},\boldsymbol{\phi}_{2,j},\cdots,\boldsymbol{\phi}_{N,j}\right\},
$$
where the $(4i-2)$-th entry of $\boldsymbol{\phi_}{i,j}$ is represented by $\boldsymbol{\zeta}_j$, denoted as
\begin{equation}\label{oddphi}
\boldsymbol{\phi_}{i,j}=\left(0,\cdots,0,\boldsymbol{\zeta}_j,0,\cdots,0\right)^T,\ 1\leq i\leq {N},\ 1\leq j \leq 3.\nonumber
\end{equation}
Moreover,
$$
\operatorname{ker} \mathcal{A}^*_0=
\operatorname{span}\left\{\boldsymbol{\varphi}_{1,j},\boldsymbol{\varphi}_{2,j},\cdots,\boldsymbol{\varphi}_{N,j}\right\},
$$
where the $(4i-2)$-th and $4i$-th entries are represented by $\boldsymbol{\xi}_j$, denoted as
$$
\quad\quad  \boldsymbol{\varphi_}{i,j}=(0,\cdots,0, \varsigma_{4i-5}^{j},0,\varsigma_{4i-3}^{j},\boldsymbol{\xi}_j,\varsigma_{4i-1}^{j},\boldsymbol{\xi}_j,\varsigma_{4i+1}^{j},0,\cdots,0)^T,\ 1< i< {N-1},\ 1\leq j \leq 3.
$$
Here, the functions $\varsigma_{4i-5}^{j},\varsigma_{4i-3}^{j}$ and $\varsigma_{4i-1}^{j},\varsigma_{4i+1}^{j}$ satisfy the following system
\begin{equation}\label{Dirichlet systems1}
\begin{cases}
\hat{\mathbf{S}}_{\Gamma_{i-1}^{-}}^{\omega,*}[\varsigma_{4i-5}^{j}]-\hat{\mathbf{S}}_{\Gamma_{i}^{+}}^{\omega,*}[\varsigma_{4i-3}^{j}]=0& \text{on}\ \Gamma_{i-1}^{-},\\
\hat{\mathbf{S}}_{\Gamma_{i-1}^{-}}^{\omega,*}[\varsigma_{4i-5}^{j}]-\hat{\mathbf{S}}_{\Gamma_{i}^{+}}^{\omega,*}[\varsigma_{4i-3}^{j}]=\boldsymbol{\xi}_{j}& \text{on}\ \Gamma_{i}^{+},
\end{cases}\quad
\begin{cases}
\hat{\mathbf{S}}_{\Gamma_{i}^{-}}^{\omega,*}[\varsigma_{4i-1}^{j}]-\hat{\mathbf{S}}_{\Gamma_{i+1}^{+}}^{\omega,*}[\varsigma_{4i+1}^{j}]=0& \text{on}\ \Gamma_{i}^{-},\\
\hat{\mathbf{S}}_{\Gamma_{i}^{-}}^{\omega,*}[\varsigma_{4i-1}^{j}]-\hat{\mathbf{S}}_{\Gamma_{i+1}^{+}}^{\omega,*}[\varsigma_{4i+1}^{j}]=-\boldsymbol{\xi}_{j}& \text{on}\ \Gamma_{i+1}^{+}.
\end{cases}\nonumber
\end{equation}
Especially for $N=1$, $\boldsymbol{\varphi_}{1,j}=(\varsigma_{1}^{j},\boldsymbol{\xi}_j,0,\boldsymbol{\xi}_j)^T$, and
$$
\boldsymbol{\varphi_}{1,j}=(\varsigma_{1}^{j},\boldsymbol{\xi}_j,\varsigma_{3}^{j},\boldsymbol{\xi}_j,\varsigma_{5}^{j},0,\cdots,0)^T,\ N\geq2,
$$
with $\varsigma_{1}^{j}=-(\hat{\mathbf{S}}_{\Gamma_{1}^{+}}^{\omega,*})^{-1}[\boldsymbol{\xi}_j]$, and
\begin{equation}
\begin{cases}
\hat{\mathbf{S}}_{\Gamma_{1}^{-}}^{\omega,*}[\varsigma_{3}^{j}]-\hat{\mathbf{S}}_{\Gamma_{2}^{+}}^{\omega,*}[\varsigma_{5}^{j}]=0& \text{on}\ \Gamma_{1}^{-},\\
\hat{\mathbf{S}}_{\Gamma_{1}^{-}}^{\omega,*}[\varsigma_{3}^{j}]-\hat{\mathbf{S}}_{\Gamma_{2}^{+}}^{\omega,*}[\varsigma_{5}^{j}]=-\boldsymbol{\xi}_{j}& \text{on}\ \Gamma_{2}^{+}.
\end{cases}\nonumber
\end{equation}
In addition, $\boldsymbol{\varphi_}{N,j}=(0,\cdots,0,\varsigma_{4N-5}^{j},0,\varsigma_{4N-3}^{j},\boldsymbol{\xi}_j,0,\boldsymbol{\xi}_j)^T$ with
\begin{equation}
\begin{cases}
\hat{\mathbf{S}}_{\Gamma_{N-1}^{-}}^{\omega,*}[\varsigma_{4N-5}^{j}]-\hat{\mathbf{S}}_{\Gamma_{N}^{+}}^{\omega,*}[\varsigma_{4N-3}^{j}]=0&\text{on}\ \Gamma_{N-1}^{-},\\
\hat{\mathbf{S}}_{\Gamma_{N-1}^{-}}^{\omega,*}[\varsigma_{4N-5}^{j}]-\hat{\mathbf{S}}_{\Gamma_{N}^{+}}^{\omega,*}[\varsigma_{4N-3}^{j}]=\boldsymbol{\xi}_{j}& \text{on}\ \Gamma_{N}^{+}.
\end{cases}\nonumber
\end{equation}
\end{lem}

With the help of matrix $\mathcal{A}(\omega,\delta)$ defined in \eqref{block diagonal2} and Lemma \ref{ker space}, we now present the proof of Theorem \ref{multifrequencies}.
\begin{proof}[Proof of Theorem \ref{multifrequencies}]
	Using the same approach of Theorem \ref{thmomega}, we can directly obtain $(\mathcal{A}_{1,0}[\Phi_0],\Psi_0)$, $(\mathcal{A}_{2,0}[\Phi_0],\Psi_0)$ and $(\mathcal{A}_{0,1}[\Phi_0],\Psi_0)$, where operators $\mathcal{A}_{1,0}$, $\mathcal{A}_{2,0}$, and $\mathcal{A}_{0,1}$ are in the form of triple diagonal blocks such as $\mathcal{A}_0$ and $\mathcal{A}^{*}_0$, we ignore them here to avoid repetition, and
	\begin{equation}
		\Phi_0=\sum_{i=1}^{N}\sum_{j=1}^3 c_{i,j}\boldsymbol{\phi}_{i,j},\quad \Psi_0=\sum_{i=1}^{N}\sum_{j=1}^3 c_{i,j}\boldsymbol{\varphi}_{i,j}.	\notag
	\end{equation}
Moreover, we introduce the following matrices
	\begin{equation}\label{matriceshat1}
		\hat{\mathbf{P}}=(\hat{P}_{ij})_{i,j=1}^{3},\quad\quad \hat{\mathbf{M}}=(\hat{M}_{ij})_{i,j=1}^{3},\quad\quad \hat{\mathbf{Q}}=(\hat{Q}_{ij})_{i,j=1}^{3}, %\nonumber
	\end{equation}
	where the $N$-by-$N$ matrices $\hat{P}_{ij}$, $\hat{M}_{ij}$ and $\hat{Q}_{ij}$ are given by
	\begin{equation}\label{matriceshat2}
		\hat{P}_{ij}=(\hat{P}_{ij}^{mn})_{mn=1}^{N},\quad\hat{M}_{ij}=(\hat{M}_{ij}^{mn})_{mn=1}^{N},\quad\hat{Q}_{ij}=(\hat{Q}_{ij}^{mn})_{mn=1}^{N}.%\nonumber
	\end{equation}
	%Precisely,
	%\begin{equation}\label{matrixQ1}
	%\hat{P}_{ij}^{mn}=\Big(\mathcal{A}_{1,0}[\boldsymbol{\phi}_{m,i}],\boldsymbol{\varphi}_{n,j}\Big),\quad\quad
	%\hat{M}_{ij}^{mn}=\Big(\mathcal{A}_{2,0}[\boldsymbol{\phi}_{m,i}],\boldsymbol{\varphi}_{n,j}\Big),\quad\quad
	%\hat{Q}_{ij}^{mn}=\Big(\mathcal{A}_{0,1}[\boldsymbol{\phi}_{m,i}],\boldsymbol{\varphi}_{n,j}\Big),\quad\quad.
	%\end{equation}
	Precisely,
	\begin{equation}\label{matriceshat3}
		\hat{P}_{ij}^{mn}=
		\begin{cases}
			\left(\mathbf{K}_{\Gamma_{m}^{+},1}^{*}[\boldsymbol{\zeta}_i],\boldsymbol{\xi}_{j}\right)_{\Gamma_{m}^{+}}-
			\left(\mathbf{K}_{\Gamma_{m}^{+},1}^{*}[\boldsymbol{\zeta}_i],\boldsymbol{\xi}_{j}\right)_{\Gamma_{m}^{-}},\ n=m,\\
			0,\quad\quad \quad\quad \quad\quad \quad\quad \quad\quad \quad\quad \quad\quad \quad\quad \ \ \ \, n\neq m ,\nonumber
		\end{cases}
	\end{equation}
	\begin{equation}\label{matriceshat4}
		\hat{M}_{ij}^{mn}=
		\begin{cases}
			\left(\mathbf{K}_{\Gamma_{m}^{+},2}^{*}[\boldsymbol{\zeta}_i],\boldsymbol{\xi}_{j}\right)_{\Gamma_{m}^{+}}-
			\left(\mathbf{K}_{\Gamma_{m}^{+},2}^{*}[\boldsymbol{\zeta}_i],\boldsymbol{\xi}_{j}\right)_{\Gamma_{m}^{-}},\ n=m,\\
			0,\quad\quad \quad\quad \quad\quad \quad\quad \quad\quad \quad\quad \quad\quad \quad\quad \ \ \ \, n\neq m ,\nonumber
		\end{cases}
	\end{equation}
	and
	\begin{equation}\label{matriceshat5}
		\hat{Q}_{ij}^{mn}=
		\begin{cases}
			\left(\hat{\mathbf{S}}_{\Gamma_{m}^{+}}^{\tau\omega}[\boldsymbol{\zeta}_i],\varsigma_{4m-3}^{j}\right)_{\Gamma_{m}^{+}},\quad \quad\quad \quad \quad\quad\quad \quad\quad\quad  n=m-1,\\
			\left(\hat{\mathbf{S}}_{\Gamma_{m}^{+}}^{\tau\omega}[\boldsymbol{\zeta}_i],\varsigma_{4m-3}^{j}\right)_{\Gamma_{m}^{+}}-
			\left(\hat{\mathbf{S}}_{\Gamma_{m}^{+}}^{\tau\omega}[\boldsymbol{\zeta}_i],\varsigma_{4m-1}^{j}\right)_{\Gamma_{m}^{-}},\ n=m,\\
			-\left(\hat{\mathbf{S}}_{\Gamma_{m}^{+}}^{\tau\omega}[\boldsymbol{\zeta}_i],\varsigma_{4m-1}^{j}\right)_{\Gamma_{m}^{-}}, \quad\quad \quad \quad\quad\quad \quad\quad\quad\, n=m+1,\\
			0,\quad\quad\,\quad \quad\quad \quad \quad\quad\quad \quad\quad\quad \quad\quad \quad \quad\quad\quad \,\text{else}.\nonumber
		\end{cases}
	\end{equation}
	Then, the occurrence of subwavelength resonance is equivalent to
	\begin{align}
		&\mathbf{c}^{T}\Big[  \rho\tau^2\omega^2\ln\omega\hat{\mathbf{P}}+\rho\tau^2\omega^2\Big(\ln{(\sqrt{\rho}\tau)}\hat{\mathbf{P}}
		+\hat{\mathbf{M}}\Big)-\delta\hat{\mathbf{Q}}\Big] \mathbf{\Lambda}\mathbf{c}+\mathcal{O}(\delta\omega^2\ln{\omega}+\omega^4\ln{\omega})=0.\nonumber
	\end{align}
	The proof is completed.
\end{proof}
%\begin{rem}
%The point scatterer approximation of the multi-resonator system \eqref{Lame system2} can be proved similarly to Theorem \ref{pointscattererapproximation}, so we skip it here.
%\end{rem}
\subsection{The point scatterer approximation}%\section{The point scatterer approximation}
The aim of this subsection is to analyse the field behaviours of system \eqref{Lame system2} when the incident frequency $\omega$ is located in different regimes. We first take a time- harmonic compressed plane wave incident field of the form
\begin{equation}\label{ui}
	\mathbf{u}^{i}(x)=\mathbf{d}e^{\mathrm{i}k_p\boldsymbol{x}\cdot\mathbf{d}},
\end{equation}
where $k_p$ is defined in \eqref{kskp} as the wavenumber of the p-wave, and $\mathbf{d}=(d_1,d_2)^{T}\in\mathbb{R}^{2}$ satisfies $\mathbf{d}\cdot\mathbf{d}=1$. In what follows, we provide the proof of Theorem \ref{pointscattererapproximation}

\begin{proof}[Proof of Theorem \ref{pointscattererapproximation}]
We only consider the single resonator case~(see subsection~3.1), while the $N$-nested resonators case can be treated similarly. First, we consider the case in which $\omega$ is far from $\omega_k$. In this regime, the operator $\mathcal{A}(\omega,\delta)$ admits a uniformly bounded inverse. Consequently, $\mathbf{u}_D=\mathcal{O}(1)$ and the coefficients $\varrho_i=\mathcal{O}(1)$.
We now restrict our attention to the regime $\delta\ln\omega\ll 1$, where the resonance frequencies $\omega_i$ are in fact situated by Remark \ref{regime_rem}. In fact, with the help of the operator $\hat{\mathcal{A}}_{0}$ defined in \eqref{hatA0}, the equation \eqref{Aomegadelta} is equivalent to
\begin{equation}\label{Operator equation}
	(\hat{\mathcal{A}}_{0}+\mathcal{B}-\mathcal{P})[\Phi](\mathbf{x})=\mathcal{F}(\mathbf{x}), \quad \mathbf{x} \in \partial D,
\end{equation}
where $\Phi\in\mathcal{H}$. Firstly, it is assumed that the function $\Phi$ has the following orthogonal decomposition%can be written as
$$
\Phi=\Phi_k+\Phi_k^{\bot}\quad \text{with}\quad (\Phi_{k},\Phi_k^{\bot})_{\mathcal{H}}=0.
$$
Here, $\Phi_k\in\ker\mathcal{A}_0$ written as a linear combination of the basis functions, i.e. $\Phi_k=\sum_{i=1}^3\alpha_i\hat\Phi_i$, where the coefficients $\alpha_i(1\leq i\leq3)$ shall be determined later. Since the operator $\hat{\mathcal{A}}_0+\mathcal{B}$ is invertible, taking $(\hat{\mathcal{A}}_0+\mathcal{B})^{-1}$ on both sides of equation \eqref{Operator equation} yields that
\begin{equation}
	\Phi_k+\Phi_k^{\bot}-\left(\mathbf{I}_2+\hat{\mathcal{A}}_0^{-1}\mathcal{B}\right)^{-1}[\Phi_k]=(\hat{\mathcal{A}}_0+\mathcal{B})^{-1}[\mathcal{F}], \quad \mathbf{x} \in \partial D.	\notag
\end{equation}
Further calculation shows that for $\mathbf{x} \in \partial D$,
\begin{equation}\label{Fequality}
	\Phi_k^{\bot}+\hat{\mathcal{A}}_0^{-1}\left(\omega^2\ln\omega\mathcal{A}_{1,0}+\omega^2\mathcal{A}_{2,0}+\delta\mathcal{A}_{0,1}
	\right)[\Phi_k]+\mathcal{O}(\delta\omega^2\ln{\omega}+\omega^4\ln{\omega})=(\hat{\mathcal{A}}_0+\mathcal{B})^{-1}[\mathcal{F}].
\end{equation}
In fact, we decompose the source term $\mathcal{F}$ as $\mathcal{F}=\mathcal{F}_0+\mathcal{O}(\omega)$, where $\mathcal{F}_0=(\mathbf{d},0)^T$. Further assume that $\mathcal{F}_0$ satisfies
$$
\mathcal{F}_0=\mathcal{F}_1+\sum_{i=1}^{3}\Upsilon_{i}\hat{\Psi}_{i}\quad \text{with}\quad \Upsilon_{i}=(\mathcal{F}_0,\hat{\Psi}_{i}).
$$
Subsequently, taking the inner product of both sides of \eqref{Fequality} with $\hat\Phi_j$ gives
\begin{align}\label{order}
	&\mathbf{\Lambda}\Big(\rho\tau^2\omega^2\ln\omega \mathbf{P}+\rho\tau^2\omega^2\left(\ln{(\sqrt{\rho}\tau)}\mathbf{P}+\mathbf{M}\right)
	-\delta \mathbf{Q}\Big)\boldsymbol{\alpha}\nonumber\\
	&+\mathcal{O}(\delta\omega^2\ln{\omega}+\omega^4\ln{\omega})=\mathbf{\Upsilon}+\mathcal{O}(\omega+\delta),
\end{align}
where we use the fact that $\Big({\hat{\mathcal{A}}}_0^{-1}[\mathcal{F}_{1}],\hat\Phi_j\Big)=0$, and $\boldsymbol{\alpha}=(\alpha_1,\alpha_2,\alpha_2)^T$, $\mathbf{\Upsilon}=(\Upsilon_{1},\Upsilon_{2},\Upsilon_{3})^T$. From \eqref{Fequality} one has
\begin{equation}\label{Phi_bot1}
	\Phi_k^{\bot}=(\hat{\mathcal{A}}_0)^{-1}[\mathcal{F}_0]-\sum_{i=1}^3 \alpha_i\hat{\mathcal{A}}_0^{-1}\left(\omega^2\ln\omega\mathcal{A}_{1,0}+\omega^2\mathcal{A}_{2,0}+\delta\mathcal{A}_{0,1}
	\right)[\hat{\Phi}_i]+o(1).
\end{equation}
Denote $\hat{\mathcal{P}}$ is the projection operator to $(\mathrm{ker}\mathcal{A}_0)^{\bot}$, taking $\hat{\mathcal{P}}$ on both sides of \eqref{Phi_bot1} we get
\begin{equation}
	\Phi_k^{\bot}=(\hat{\mathcal{A}}_0)^{-1}[\mathcal{F}_1]-\sum_{i=1}^3 \alpha_i \hat{\mathcal{P}} \hat{\mathcal{A}}_0^{-1}\left(\omega^2\ln\omega\mathcal{A}_{1,0}+\omega^2\mathcal{A}_{2,0}+\delta\mathcal{A}_{0,1}
	\right)[\hat{\Phi}_i]+o(1).
\end{equation}
Since $\mathcal{F}_1 \notin \mathrm{ker}\mathcal{A}_0$, which says that $\mathcal{F}_1\in \mathrm{Range}(\mathcal{A}_0)$, thus $(\hat{\mathcal{A}}_0)^{-1}[\mathcal{F}_1]={\mathcal{A}}_0^{-1}[\mathcal{F}_1]$. By the definition of $\mathcal{A}_0$ we can immediate get
\begin{equation}
	({\mathcal{A}}_0^{-1}[\mathcal{F}_1])_2=-(\hat{\mathbf{S}}_{\partial D}^{\omega})^{-1}[(\mathcal{F}_1)_1].
\end{equation}
where $({\mathcal{A}}_0^{-1}[\mathcal{F}_1])_2$ signifying the second component of the vector ${\mathcal{A}}_0^{-1}[\mathcal{F}_1]$.
Indeed, the following formula holds
\begin{equation}
	\int_{\partial D}({\mathcal{A}}_0^{-1}[\mathcal{F}_1])_2\cdot \boldsymbol{\xi}_j \mathrm{d}x=-\int_{\partial D}(\mathcal{F}_1)_1\cdot (\hat{\mathbf{S}}_{\partial D}^{\omega,*})^{-1}[\boldsymbol{\xi}_j] \mathrm{d}x=-(\mathcal{F}_1,\hat{\Psi}_j)_{\mathcal{H}}=0, \notag
\end{equation}
the derivation relies on the fact $(\mathcal{F}_1)_2\in L^2_{\mathbf{\Psi}}(\partial D)$. Similarly, one has
\begin{equation}
	\int_{\partial D}\left(\frac{\mathbf{I}_2}{2}+\mathbf{K}_{\partial D}^{*}\right)(\hat{\mathbf{S}}_{\partial D}^{\omega})^{-1}[(\mathcal{F}_1)_1]\cdot \boldsymbol{\xi}_j \mathrm{d}x =0. \notag
\end{equation}
Thanks to the invertibility of $-\frac{\mathbf{I}_2}{2}+\mathbf{K}_{\partial D}^{*}$ in $L^2_{\mathbf{\Psi}}(\partial D)$, then
\begin{equation}
	({\mathcal{A}}_0^{-1}[\mathcal{F}_1])_1=\left(-\frac{\mathbf{I}_2}{2}+\mathbf{K}_{\partial D}^{*}\right)^{-1}\left[-\left(\frac{\mathbf{I}_2}{2}+\mathbf{K}_{\partial D}^{*}\right)(\hat{\mathbf{S}}_{\partial D}^{\omega})^{-1}[(\mathcal{F}_1)_1]+(\mathcal{F}_1)_2\right], \notag
\end{equation}
also belongs to $L^2_{\mathbf{\Psi}}(\partial D)$. At this moment, the displacement field inside the domain $D$ is given by
\begin{equation}
	\begin{aligned}
		\tilde{\mathbf{S}}_{\partial D}^{\omega}[\Phi]&=\delta{\hat{\mathbf{S}}}_{\partial D}^{\tau\omega}[\Phi]+\mathcal{O}(\delta\omega^2\ln\omega)=\sum_{i=1}^3\alpha_i\delta \hat{\mathbf{S}}^{\tau\omega}_{\partial D}[\hat\Phi_i]+\delta{\hat{\mathbf{S}}}_{\partial D}^{\tau\omega}[\Phi_k^{\bot}]+\mathcal{O}(\delta\omega^2\ln\omega)\\
		&=\delta\left(\sum_{i=1}^3\alpha_i\hat{\mathbf{S}}^{\tau\omega}_{\partial D}[\hat\Phi_i](1+\mathcal{O}(\omega^2\ln\omega+\delta))+{{\mathbf{S}}}_{\partial D}^{\tau\omega}\Big[\mathcal{A}_0^{-1}[\mathcal{F}_1]\Big]\right)+\mathcal{O}(\delta\omega^2\ln\omega).
	\end{aligned}\notag
\end{equation}
Thus, when $\omega$ satisfies \eqref{formula1}, one has that $\alpha=\mathcal{O}(1/(\delta\omega^{2}\ln \omega))$ from \eqref{order} and the displacement field inside the domain $D$ takes the form
\vspace{-4pt}
\begin{equation}\label{uD}
	\begin{aligned}
		\mathbf{u}_{D}&=\sum_{i=1}^3\delta\alpha_i \left(\mathbf{S}_{\partial D}[\boldsymbol{\zeta}_i]+\int_{\partial D}\gamma_{\tau\omega}\boldsymbol{\zeta}_i+\mathcal{O}(\omega^2\ln\omega+\delta)\right)+\mathcal{O}(\delta)\\
		&=\sum_{i=1}^3\varrho_i \boldsymbol{\xi}_i+\mathcal{O}(1)=\mathcal{O}(\omega^{-2}),
	\end{aligned}
\end{equation}
since $\gamma_{\tau\omega}=\mathcal{O}(\ln\omega)$. When $\omega^2\gg \delta$, one has that $\alpha=\mathcal{O}(1/(\omega^{2}\ln \omega))$, thus the displacement field inside the domain $D$ satisfies
\begin{equation}
	\mathbf{u}_{D}=\sum_{i=1}^3\delta\alpha_i \left(\mathbf{S}_{\partial D}[\boldsymbol{\zeta}_i]+\int_{\partial D}\gamma_{\tau\omega}\boldsymbol{\zeta}_i+\mathcal{O}(\omega^2\ln\omega+\delta)\right)+\mathcal{O}(\delta)=\mathcal{O}(\delta/\omega^2).
\end{equation}
The proof is completed by noting that $\omega\ll 1$, so that $\varrho_i=o(1)$ and the term of $\mathcal{O}(\delta)$ can be regarded as higher order term.
\end{proof}
\section{The sub-wavelength resonances for disk geometry}
In this section, we use two methods to calculate the explicit expressions of the resonance frequencies within a disk. One is to utilize the conclusion derived from the formula \eqref{formula1} in Theorem \ref{thmomega}, the other is by directly solving \eqref{Aomegadelta}. Although both analyses rely on asymptotic expansions in the frequency parameter $\omega$. To aid readers, we start by introducing the Bessel and Hankel functions, as well as some of the spectral properties of N-P operators in one disk geometry.
\subsection{The Bessel and Hankel functions}
Let $J_n(t)$ and $H_n(t),n\in\mathbb{Z}$, denote the Bessel and Hankel functions of the first kind of order $n$, respectively. For $f=J_{n}$ or $H_{n}$, these functions satisfy the following Bessel differential equation
$$
t^2f''(t)+tf'(t)+\left(t^2-n^2\right)f(t)=0,
$$
and the recursion formulas (cf. \cite{ColtonKress})
$$
f_{n+1}^{\prime}=f_{n}-(n+1) \frac{f_{n+1}}{t}, \quad f_{n+1}=n \frac{f_{n}}{t}-f_{n}^{\prime}, \quad n \geq 0,
$$
$$
f_{n-1}^{\prime}=-f_{n}+(n-1) \frac{f_{n-1}}{t}, \quad f_{n-1}=n \frac{f_{n}}{t}+f_{n}^{\prime}, \quad n \geq 1.
$$
Moreover, there hold that $J_{n}(t)=(-1)^{n} J_{-n}(t)$ and $H_{n}(t)= (-1)^{n} H_{-n}(t)$ when $n$ is negative. Therefore, for the argument $t \ll 1$, the functions $J_n$ and $H_n$, enjoy the following asymptotic expansions(cf. \cite{ColtonKress}),
\begin{equation}\label{J0}
J_0(t)=1-\frac{t^2}{4}+\frac{t^4}{64}-\frac{t^6}{2^8\cdot9}+\mathcal{O}(t^8),
\end{equation}
\begin{equation}\label{J1}
J_1(t)=\frac{t}{2}-\frac{t^3}{16}+\frac{t^5}{2^7\cdot3}+\mathcal{O}(t^7),
\end{equation}
and
\begin{equation}\label{H0}
H_0(t)=\frac{\mathrm{i}}{\pi}\left(E_c+2\ln t\right)-\frac{\mathrm{i}t^2}{4\pi}\left(-2+E_c+2\ln t\right)
+\frac{\mathrm{i}t^4}{64\pi}\left(-3+E_c+2\ln t\right)
+\mathcal{O}\left(t^6\ln t\right),
\end{equation}
\begin{equation}\label{H1}
H_1(t)=-\frac{2\mathrm{i}}{\pi t}+\frac{\mathrm{i}t}{2\pi}\left(-1+E_c+2\ln t\right)
-\frac{\mathrm{i}t^3}{16\pi}\left(-\frac{5}{2}+E_c+2\ln t\right)
+\mathcal{O}\left(t^5\ln t\right).
\end{equation}

Let $D_{R}\subset\mathbb{R}^{2}$ denote the disk centered at the origin with radius $R\in \mathbb{R}_{+}$, $\mathbf{x}=\left(x_{1},x_{2}\right)\in \mathbb{R}^{2}$ is the Euclidean coordinate, and $\theta$ is the angle between $\mathbf{x}$ and the $x_{1}$-axis, $r=|\mathbf{x}|$. In the following part, we can first derive the expressions for the single-layer potentials $\mathbf{S}_{\partial D_{R}}^{\omega}$ associated with the two densities $e^{\mathrm{i} n \theta} \boldsymbol{v}$ and $e^{\mathrm{i} n \theta} \boldsymbol{t}$, where $\boldsymbol{v}=(\cos\theta, \sin \theta)^{T}$, $\boldsymbol{t}=(-\sin\theta, \cos\theta)^{T}$ represent the outward unit normal and tangential direction along a boundary $\partial D_{R}$, respectively.~(cf. \cite{2Delasticspectrum}).
\begin{lem}\label{Spectrum1}
The single-layer potentials $\mathbf{S}_{\partial D_{R}}^{\omega}\left[e^{\mathrm{i} n \theta} \boldsymbol{v}\right]$ and $\mathbf{S}_{\partial D_{R}}^{\omega}\left[e^{\mathrm{i} n \theta} \boldsymbol{t}\right]$ have the following expressions for $\mathbf{x} \in \partial D_{R}$,
$$
\mathbf{S}_{\partial D_{R}}^{\omega}\left[e^{\mathrm{i} n \theta} \boldsymbol{v}\right](\mathbf{x}) =\alpha_{1 n} e^{\mathrm{i} n \theta} \boldsymbol{v}+\alpha_{2 n} e^{\mathrm{i} n \theta} \boldsymbol{t},
$$
$$
\mathbf{S}_{\partial D_{R}}^{\omega}\left[e^{\mathrm{i} n \theta} \boldsymbol{t}\right](\mathbf{x}) =\alpha_{3 n} e^{\mathrm{i} n \theta} \boldsymbol{v}+\alpha_{4 n} e^{\mathrm{i} n \theta} \mathbf{t},
$$
where
$$
\alpha_{1 n}=-\frac{\mathrm{i} \pi}{2 \omega^{2} \rho R}\left(n^{2} J_{n}\left(k_{s} R\right) H_{n}\left(k_{s} R\right)+k_{p}^{2} R^{2} J_{n}^{\prime}\left(k_{p} R\right) H_{n}^{\prime}\left(k_{p} R\right)\right),
$$
$$
\alpha_{2 n}=\frac{n \pi}{2 \omega^{2} \rho}\left(k_{s} J_{n}\left(k_{s} R\right) H_{n}^{\prime}\left(k_{s} R\right)+k_{p} J_{n}^{\prime}\left(k_{p} R\right) H_{n}\left(k_{p} R\right)\right),
$$
$$
\alpha_{3 n}=-\frac{n \pi}{2 \omega^{2} \rho}\left(k_{s} J_{n}^{\prime}\left(k_{s} R\right) H_{n}\left(k_{s} R\right)+k_{p} J_{n}\left(k_{p} R\right) H_{n}^{\prime}\left(k_{p} R\right)\right),
$$
$$
\alpha_{4 n}=-\frac{\mathrm{i} \pi}{2 \omega^{2} \rho R}\left(k_{s}^{2} R^{2} J_{n}^{\prime}\left(k_{s} R\right) H_{n}^{\prime}\left(k_{s} R\right)+n^{2} J_{n}\left(k_{p} R\right) H_{n}\left(k_{p} R\right)\right).
$$
\end{lem}
The single-layer potentials $\mathbf{S}_{\partial D_{R}}^{\omega}\left[e^{\mathrm{i} n \theta} \boldsymbol{v}\right]$ and $\mathbf{S}_{\partial D_{R}}^{\omega}\left[e^{\mathrm{i} n \theta} \boldsymbol{t}\right]$ have the following expressions for $\mathbf{x} \in \mathbb{R}^{2} \backslash \bar{D}_{R}$,
$$
\mathbf{S}_{\partial D_{R}}^{\omega}\left[e^{\mathrm{i} n \theta} \boldsymbol{v}\right](\mathbf{x})=\frac{-\mathrm{i} \pi}{4 \omega^{2} \rho R}\left(n k_{s} R J_{n}\left(k_{s} R\right) \boldsymbol{\Psi}_{n}^{s, o}\left(k_{s}|\mathbf{x}|\right)+k_{p}^{2} R^{2} J_{n}^{\prime}\left(k_{p} R\right) \boldsymbol{\Psi}_{n}^{p, o}\left(k_{p}|\mathbf{x}|\right)\right),
$$
$$
\mathbf{S}_{\partial D_{R}}^{\omega}\left[e^{\mathrm{i} n \theta} \mathbf{t}\right](\mathbf{x})=\frac{-\pi}{4 \omega^{2} \rho R}\left(k_{s}^{2} R^{2} J_{n}^{\prime}\left(k_{s} R\right) \boldsymbol{\Psi}_{n}^{s, o}\left(k_{s}|\mathbf{x}|\right)+n k_{p} R J_{n}\left(k_{p} R\right) \boldsymbol{\Psi}_{n}^{p, o}\left(k_{p}|\mathbf{x}|\right)\right),
$$
where
$$
\boldsymbol{\Psi}_{n}^{s, o}\left(k_{s}|\mathbf{x}|\right)=\frac{2 n H_{n}\left(k_{s}|\mathbf{x}|\right)}{k_{s}|\mathbf{x}|} e^{\mathrm{i} n \theta} \boldsymbol{v}+2 \mathrm{i} H_{n}^{\prime}\left(k_{s}|\mathbf{x}|\right) e^{\mathrm{i} n \theta} \boldsymbol{t},
$$
$$
\boldsymbol{\Psi}_{n}^{p, o}\left(k_{p}|\mathbf{x}|\right)=2 H_{n}^{\prime}\left(k_{p}|\mathbf{x}|\right) e^{\mathrm{i} n \theta} \boldsymbol{v}+\frac{2 \mathrm{i} n H_{n}\left(k_{p}|\mathbf{x}|\right)}{k_{p}|\mathbf{x}|} e^{\mathrm{i} n \theta} \boldsymbol{t}.
$$
Besides, the single-layer potentials $\mathbf{S}_{\partial D_{R}}^{\omega}\left[e^{\mathrm{i} n \theta} \boldsymbol{v}\right]$ and $\mathbf{S}_{\partial D_{R}}^{\omega}\left[e^{\mathrm{i} n \theta} \boldsymbol{t}\right]$ have the following expressions for $\mathbf{x} \in D_{R}$,
$$
\mathbf{S}_{\partial D_{R}}^{\omega}\left[e^{\mathrm{i} n \theta} \boldsymbol{v}\right](\mathbf{x})=\frac{-\mathrm{i} \pi}{4 \omega^{2} \rho R}\left(n k_{s} R H_{n}\left(k_{s} R\right) \boldsymbol{\Psi}_{n}^{s, i}\left(k_{s}|\mathbf{x}|\right)+k_{p}^{2} R^{2} H_{n}^{\prime}\left(k_{p} R\right) \boldsymbol{\Psi}_{n}^{p, i}\left(k_{p}|\mathbf{x}|\right)\right),
$$
$$
\mathbf{S}_{\partial D_{R}}^{\omega}\left[e^{\mathrm{i} n \theta} \boldsymbol{t}\right](\mathbf{x})=\frac{-\pi}{4 \omega^{2} \rho R}\left(k_{s}^{2} R^{2} H_{n}^{\prime}\left(k_{s} R\right) \boldsymbol{\Psi}_{n}^{s, i}\left(k_{s}|\boldsymbol{x}|\right)+n k_{p} R H_{n}\left(k_{p} R\right) \boldsymbol{\Psi}_{n}^{p, i}\left(k_{p}|\boldsymbol{x}|\right)\right),
$$
where
$$
\boldsymbol{\Psi}_{n}^{s, i}\left(k_{s}|\boldsymbol{x}|\right)=\frac{2 n J_{n}\left(k_{s}|\boldsymbol{x}|\right)}{k_{s}|\boldsymbol{x}|} e^{\mathrm{i} n \theta} \boldsymbol{v}+2 \mathrm{i} J_{n}^{\prime}\left(k_{s}|\boldsymbol{x}|\right) e^{\mathrm{i} n \theta} \boldsymbol{t},
$$
$$
\boldsymbol{\Psi}_{n}^{p, i}\left(k_{p}|\boldsymbol{x}|\right)=2 J_{n}^{\prime}\left(k_{p}|\boldsymbol{x}|\right) e^{\mathrm{i} n \theta} \boldsymbol{v}+\frac{2 \mathrm{i} n J_{n}\left(k_{p}|\boldsymbol{x}|\right)}{k_{p}|\boldsymbol{x}|} e^{\mathrm{i} n \theta} \boldsymbol{t}.
$$

\begin{lem}\label{Spectrum2}
The N-P operator $\mathbf{K}_{\partial D_{R}}^{\omega, *}$ have the following expressions with two densities $e^{\mathrm{i} n \theta} \boldsymbol{v}$ and $e^{\mathrm{i} n \theta} \boldsymbol{t}$,
$$
\mathbf{K}_{\partial D_{R}}^{\omega, *}\left[e^{\mathrm{i} n \theta} \boldsymbol{v}\right]=a_{1 n} e^{\mathrm{i} n \theta} \boldsymbol{v}+a_{2 n} e^{\mathrm{i} n \theta} \boldsymbol{t},
$$
$$
\mathbf{K}_{\partial D_{R}}^{\omega, *}\left[e^{\mathrm{i} n \theta} \boldsymbol{t}\right] =b_{1 n} e^{\mathrm{i} n \theta} \boldsymbol{v}+b_{2 n} e^{\mathrm{i} n \theta} \boldsymbol{t},
$$
where
$$
a_{1 n}=-\frac{1}{2}+g_{1, n}(R), \quad a_{2 n}=g_{2, n}(R), \quad b_{1 n}=g_{3, n}(R), \quad b_{2 n}=-\frac{1}{2}+g_{4, n}(R).
$$
The coefficients $g_{i, n}(R)(1 \leq i \leq 4)$ can be derived from the tractions $\left.\partial_{\nu} \mathbf{S}_{\partial D_{R}}^{\omega}\left[e^{\mathrm{i} n \theta} \boldsymbol{v}\right]\right|_{+}$ and \\
$\left.\partial_{\nu} \mathbf{S}_{\partial D_{R}}^{\omega}\left[e^{\mathrm{i} n \theta} \boldsymbol{t}\right]\right|_{+}$ evaluated on $\partial D_{R}$, as defined by the following expressions
$$
\left.\partial_{\nu} \mathbf{S}_{\partial D_{R}}^{\omega}\left[e^{\mathrm{i} n \theta} \boldsymbol{v}\right]\right|_{+}=g_{1, n}(|\mathbf{x}|) e^{\mathrm{i} n \theta} \boldsymbol{v}+g_{2, n}(|\mathbf{x}|) e^{\mathrm{i} n \theta} \boldsymbol{t},
$$
$$
\left.\partial_{\nu} \mathbf{S}_{\partial D_{R}}^{\omega}\left[e^{\mathrm{i} n \theta} \boldsymbol{t}\right]\right|_{+}=g_{3, n}(|\mathbf{x}|) e^{i n \theta} \boldsymbol{v}+g_{4, n}(|\mathbf{x}|) e^{i n \theta} \boldsymbol{t},
$$
where
\begin{equation}
\begin{aligned}
g_{1, n}(|\mathbf{x}|)= & \frac{\mathrm{i} \pi}{2 \omega^{2} \rho |\mathbf{x}|^{2}}\left(2 \mu n^{2} J_{n}\left(k_{s} R\right)\left(H_{n}\left(k_{s}|\mathbf{x}|\right)-k_{s} |\mathbf{x}| H_{n}^{\prime}\left(k_{s}|\mathbf{x}|\right)\right)+\right. \\
& \left.J_{n}^{\prime}\left(k_{p} R\right) k_{p} R\left(H_{n}\left(k_{p}|\mathbf{x}|\right)\left(\omega^{2} \rho |\mathbf{x}|^{2}-2 \mu n^{2}\right)+2 k_{p} \mu |\mathbf{x}| H_{n}^{\prime}\left(k_{p}|\mathbf{x}|\right)\right)\right),\nonumber
\end{aligned}
\end{equation}
\begin{equation}
\begin{aligned}
g_{2, n}(|\mathbf{x}|)= & -\frac{n \mu \pi}{2 \omega^{2} \rho |\mathbf{x}|^{2}}\left(J_{n}\left(k_{s} R\right) H_{n}\left(k_{s}|\mathbf{x}|\right)\left(k_{s}^{2} |\mathbf{x}|^{2}-2 n^{2}\right)+2k_{s}|\mathbf{x}| J_{n}\left(k_{s} R\right) H_{n}^{\prime}\left(k_{s}|\mathbf{x}|\right)+\right. \\
& \left.2 k_{p}R J_{n}^{\prime}\left(k_{p} R\right)\left(H_{n}\left(k_{p}|\mathbf{x}|\right)-k_{p} |\mathbf{x}| H_{n}^{\prime}\left(k_{p}|\mathbf{x}|\right)\right)\right),\nonumber
\end{aligned}
\end{equation}
\begin{equation}
\begin{aligned}
g_{3, n}(|\mathbf{x}|)= & \frac{n \pi}{2 \omega^{2} \rho |\mathbf{x}|^{2}}\left(J_{n}\left(k_{p} R\right) H_{n}\left(k_{p}|\mathbf{x}|\right)\left(\omega^{2} \rho |\mathbf{x}|^{2}-2 \mu n^{2}\right)+2 \mu k_{p}|\mathbf{x}| J_{n}\left(k_{p} R\right) H_{n}^{\prime}\left(k_{p}|\mathbf{x}|\right)+\right. \\
& \left.2 \mu k_{s}R J_{n}^{\prime}\left(k_{s} R\right)\left(H_{n}\left(k_{s}|\mathbf{x}|\right)-k_{s} |\mathbf{x}| H_{n}^{\prime}\left(k_{s}|\mathbf{x}|\right)\right)\right),\nonumber
\end{aligned}
\end{equation}
\begin{equation}
\begin{aligned}
g_{4, n}(|\mathbf{x}|)= & \frac{\mathrm{i} \mu \pi}{2 \omega^{2} \rho |\mathbf{x}|^{2}}\left(2 n^{2} J_{n}\left(k_{p} R\right)\left(H_{n}\left(k_{p}|\mathbf{x}|\right)-k_{p} |\mathbf{x}| H_{n}^{\prime}\left(k_{p}|\mathbf{x}|\right)\right)+\right. \\
&\left. k_{s} R J_{n}^{\prime}\left(k_{s} R\right) \left(k_{s}^{2} |\mathbf{x}|^{2} H_{n}\left(k_{s}|\mathbf{x}|\right)+2 k_{s} |\mathbf{x}| H_{n}^{\prime}\left(k_{s}|\mathbf{x}|\right)-2 n^{2} H_{n}\left(k_{s}|\mathbf{x}|\right)\right)\right).\nonumber
\end{aligned}
\end{equation}
\end{lem}%In order to facilitate the presentation of our ideas in this part,

\subsection{The derivation of resonance frequencies from Theorem \ref{thmomega}} In this part, we obtain the specific expression for the eigenfrequency from Theorem \ref{thmomega}. Firstly, the elements of the kernel space are given by the following%~(see \cite{Subwavelengthresonances})
\begin{equation}
\boldsymbol{\zeta}_1=\frac{1}{{2\pi R}}\left(\begin{array}{l}
1  \\
0
\end{array}\right),\quad\quad
\boldsymbol{\zeta}_2=\frac{1}{{2\pi R}}\left(\begin{array}{l}
0  \\
1
\end{array}\right),\quad\quad
\boldsymbol{\zeta}_3=\frac{1}{{2\pi R^3}}\left(\begin{array}{l}
x_2  \\
-x_1
\end{array}\right).\nonumber
\end{equation}
These basis functions are orthonormal in $L^{2}(\partial D)^{2}$. Then, introducing parameters
\begin{equation}
\sigma_1=\frac{1}{2}\left(\frac{1}{\mu}+\frac{1}{2 \mu+\lambda}\right) \quad \text { and } \quad \sigma_2=\frac{1}{2}\left(\frac{1}{\mu}-\frac{1}{2 \mu+\lambda}\right).\nonumber
\end{equation}
Finally, we have the following lemma.
\begin{lem}\label{disk singlelayer}
Let $D$ be a disk of radius $R$. For $\mathbf{x}\in D$, there holds
\begin{equation}\label{disk singlelayer1}
{\mathbf{S}}_{\partial D}[\boldsymbol{\zeta}_i](\mathbf{x})=\int_{\partial D} \boldsymbol{\Gamma}(\mathbf{x}-\mathbf{y}) \boldsymbol{\zeta}_i(\mathbf{y}) d s(\mathbf{y})=\left(\sigma_1R\ln R-\frac{\sigma_2R}{2}\right)\boldsymbol{\zeta}_i(\mathbf{x}),\ i=1,2.
\end{equation}
and
\begin{equation}\label{disk singlelayer2}
{\mathbf{S}}_{\partial D}[\boldsymbol{\zeta}_3](\mathbf{x})=\int_{\partial D} \boldsymbol{\Gamma}(\mathbf{x}-\mathbf{y}) \boldsymbol{\zeta}_3(\mathbf{y}) d s(\mathbf{y})={-\frac{R}{2\mu}}\boldsymbol{\zeta}_3(\mathbf{x}).
\end{equation}
\end{lem}
\begin{proof}
Indeed, the first equation \eqref{disk singlelayer1} can be proven by applying the fact that~(cf. \cite{2018Kangelasticity})
$$
{\mathbf{S}}_{\partial D}[c](\mathbf{x})=\left(\sigma_1R\ln R-\frac{\sigma_2R}{2}\right)c, \quad \mathbf{x}\in D,
$$
for any constant vector $c=(c_1,c_2)^T$. In addition, by using polar coordinate calculations, it holds that
$$
\int_{\partial D} \frac{1}{2 \pi} \ln |\mathbf{x}-\mathbf{y}|
\begin{pmatrix}
y_2\\
-y_1
\end{pmatrix}
d s(\mathbf{y})=-\frac{R}{2}
\begin{pmatrix}
x_2\\
-x_1
\end{pmatrix},\quad \mathbf{x}\in D,
$$
and
$$
\int_{\partial D} \frac{1}{2 \pi}\frac{(\mathbf{x}-\mathbf{y})(\mathbf{x}-\mathbf{y})^T}{|\mathbf{x}-\mathbf{y}|^2}
\begin{pmatrix}
y_2\\
-y_1
\end{pmatrix}
d s(\mathbf{y})=\frac{R}{2}
\begin{pmatrix}
x_2\\
-x_1
\end{pmatrix},\quad \mathbf{x}\in D.
$$
Hence, combining these with \eqref{gamma} yields the second equation \eqref{disk singlelayer2}.
\end{proof}
In fact, let
$$
\mathbf{C}=\text{diag}\left(\frac{1}{2\pi}\Big(\sigma_1\ln R-\frac{\sigma_2}{2},\sigma_1\ln R-\frac{\sigma_2}{2},-\frac{1}{2\mu R^2}\Big)\right) ,
$$
it can verify the conclusion given in the previous result (Lemma \ref{kernelspace}) that
\begin{equation}
\mathbf{S}_{\partial D}[\mathbf{G}]=\mathbf{F}\mathbf{C},\quad \quad \int_{\partial D}\mathbf{F}(\mathbf{y})^T\mathbf{G}(\mathbf{y})d s(\mathbf{y})=\mathbf{I}_3.\nonumber
\end{equation}
Here, we need to emphasize that {$\mathbf{C}$ is singular matrix, i.e.~$\det\mathbf{C}= 0$ if $\sigma_1\ln R=\frac{\sigma_2}{2}$. At this moment, $\hat{\mathbf{S}}_{\partial D}^{\omega}$ is still invertible, but ${\mathbf{S}}_{\partial D}$ is not}.
\begin{lem}\label{thmdiskomega}
Let $D$ be a disk of radius $R$. For $\delta\ll 1$, and $\epsilon\ll 1$, there exist three sub-wavelength resonance frequencies,
counted with their multiplicities, whose leading-order terms denoted can be determined by the equations as follows
\begin{equation}\label{formuladisk}
{\rho\omega^2\ln\omega p_{i}+\rho\omega^2\Big(\ln{(\sqrt{\rho}\tau)}p_{i}
+m_{i}\Big)-\epsilon q_{i}=0, \quad i=1,2,3.}%\nonumber
\end{equation}
where $p_{i}$, $m_{i}$ and $q_{i}(1 \leq i \leq 3)$ are given by
\begin{equation}\label{Peigenvalue}
p_{1}=p_{2}={a}_{\lambda,\mu}\pi R^2, \quad\quad  p_{3}=0,\nonumber
\end{equation}
\begin{equation}\label{Meigenvalue}
m_{1}=m_{2}={b}_{\lambda,\mu}\pi R^2-\frac{R^2}{2}\left(\sigma_{1}\ln{R}-\frac{\sigma_2}{2}\right),\quad\quad m_{3}=\frac{R^2}{8\mu},\nonumber
\end{equation}
and
\begin{equation}\label{Qeigenvalue}
q_{1}=q_{2}=\frac{\sigma_{1}\ln{R}-\frac{\sigma_2}{2}+2\pi \gamma_{\tau\omega}}{\sigma_{1}\ln{R}-\frac{\sigma_2}{2}+2\pi \gamma_{\omega}},  \quad\quad  q_{3}=1.\nonumber
\end{equation}
Here, the constants ${a}_{\lambda,\mu}$ and ${b}_{\lambda,\mu}$ are given by \eqref{lamea} and \eqref{lameb}, respectively.
\end{lem}
\begin{proof}
Consider first, that $\{\boldsymbol{\xi}_i\}_{i=1}^{3}$ and $\{\boldsymbol{\zeta}_i\}_{i=1}^{3}$ are orthonormal, it is clear that
\begin{equation}
P_{ij}=0,\quad \quad M_{ij}=0, \quad \quad Q_{ij}=0, \quad\quad\text{for} \ i\neq j. \nonumber
\end{equation}
Next, it follows from Proposition \ref{matrices p and m} and Lemma \ref{disk singlelayer} that $P_{33}=p_3$,
$$
P_{ii}=a_{\lambda,\mu}\int_{ \partial D}\boldsymbol{\zeta}_i(\mathbf{x})ds(\mathbf{x})\cdot \int_{ D}\boldsymbol{\xi}_i(\mathbf{x})d\mathbf{x}
=\frac{{a}_{\lambda,\mu}}{2\pi R}|\partial D|Vol(D)=p_i,\quad i=1,2,
$$
$$
M_{33}=-\int_{ D} {\mathbf{S}}_{\partial D}[\boldsymbol{\zeta}_3](\mathbf{x}) \cdot \boldsymbol{\xi}_3(\mathbf{x}) d\mathbf{x}
=\frac{R}{2\mu}\frac{1}{2\pi R^3}\int_{D}r^2=m_3,
$$
and
\begin{align}
M_{ii}&={b}_{\lambda,\mu}\int_{ \partial D}\boldsymbol{\zeta}_i(\mathbf{x})ds(\mathbf{x}) \cdot\int_{ D}\boldsymbol{\xi}_i(\mathbf{x})d\mathbf{x}
-\int_{ D} {\mathbf{S}}_{\partial D}[\boldsymbol{\zeta}_i](\mathbf{x}) \cdot \boldsymbol{\xi}_i(\mathbf{x}) d\mathbf{x}=m_i,\quad i=1,2.\nonumber
\end{align}
In addition, by using Lemma \ref{disk singlelayer} and equations \eqref{hats1}, \eqref{hats2}, we have
$$
\hat{\mathbf{S}}_{\partial D}^{{\tau\omega} }[\boldsymbol{\zeta}_i]=R\Big(\sigma_{1}\ln{R}-\frac{\sigma_2}{2}+2\pi \gamma_{\tau\omega}\Big)\boldsymbol{\zeta}_i,\quad i=1,2,
$$
$$
\hat{\mathbf{S}}_{\partial D}^{{\omega} }[\boldsymbol{\zeta}_i]=R\Big(\sigma_{1}\ln{R}-\frac{\sigma_2}{2}+2\pi \gamma_{\omega}\Big)\boldsymbol{\zeta}_i,\quad i=1,2,
$$
and
$$
\hat{\mathbf{S}}_{\partial D}^{{\tau\omega} }[\boldsymbol{\zeta}_3]=\hat{\mathbf{S}}_{\partial D}^{{\omega} }[\boldsymbol{\zeta}_3]={-\frac{R}{2\mu}}\boldsymbol{\zeta}_3.
$$
Hence,
\begin{align}
Q_{ii}&=\Big(\hat{\mathbf{S}}_{\partial D}^{{\tau\omega} }[\boldsymbol{\zeta}_i],(\hat{\mathbf{S}}_{\partial D}^{{\omega},* })^{-1}[\boldsymbol{\xi}_i]\Big)=R\Big(\sigma_{1}\ln{R}-\frac{\sigma_2}{2}+2\pi \gamma_{\omega}\Big)\Big((\hat{\mathbf{S}}_{\partial D}^{{\omega}})^{-1}[\boldsymbol{\zeta}_i],\boldsymbol{\xi}_i\Big)=q_i,\quad i=1,2.\nonumber
\end{align}
In the same way, we can obtain that $Q_{33}=q_3$. Finally, it reaches the proof.
\end{proof}
\subsection{The derivation of resonance frequencies by solving \eqref{Aomegadelta}}In this part, we obtain the
specific expression for the eigenfrequency by directly solving the equation \eqref{Aomegadelta}. For the sake of simplicity of presentation, denote%throughout the rest of the paper,
$$
\boldsymbol{v_{n}}=e^{\mathrm{i} n \theta} \boldsymbol{v},\quad\quad \quad \boldsymbol{t_{n}}=e^{\mathrm{i} n \theta} \boldsymbol{t}.
$$
Straightforward computations show that
$$
-r\boldsymbol{t_{0}}=-r\boldsymbol{t}=\boldsymbol{\xi}_3=
\left(\begin{array}{l}
x_{2} \\
-x_{1}
\end{array}\right),
$$
$$
\frac{\boldsymbol{v_{1}}+\boldsymbol{v_{-1}}}{2}-\frac{\boldsymbol{t_{1}}-\boldsymbol{t_{-1}}}{2\mathrm{i}}=\boldsymbol{\xi}_1=
\left(\begin{array}{l}
1 \\
0
\end{array}\right),
$$
$$
\frac{\boldsymbol{v_{1}}-\boldsymbol{v_{-1}}}{2\mathrm{i}}+\frac{\boldsymbol{t_{1}}+\boldsymbol{t_{-1}}}{2}=\boldsymbol{\xi}_2=
\left(\begin{array}{l}
0 \\
1
\end{array}\right).
$$
To better illustrate the results, we introduce the following functions
\[
\boldsymbol{\Xi}_1:=\Xi_{\boldsymbol{ t_{0}}},\quad\quad  \boldsymbol{\Xi}_2:=(\Xi_{\boldsymbol{ v_{1}}},\Xi_{\boldsymbol{ t_{1}}}), \quad \quad \boldsymbol{\Xi}_3:=(\Xi_{\boldsymbol{ v_{-1}}},\Xi_{\boldsymbol{ t_{-1}}}),\nonumber
\]
where
\begin{equation}
\Xi_{\boldsymbol{ t_{0}}}=
\begin{pmatrix}
\boldsymbol{t_{0}}& 0 \\
0&\boldsymbol{t_{0}}
\end{pmatrix},\quad\quad
\Xi_{\boldsymbol{ v_{1}}}=
\begin{pmatrix}
\boldsymbol{v_{1}}& 0 \\
0&\boldsymbol{v_{1}}
\end{pmatrix},\quad\quad
\Xi_{\boldsymbol{ t_{1}}}=
\begin{pmatrix}
\boldsymbol{t_{1}}& 0 \\
0&\boldsymbol{t_{1}}
\end{pmatrix},\nonumber
\end{equation}
and
\begin{equation}
\Xi_{\boldsymbol{ v_{-1}}}=
\begin{pmatrix}
\boldsymbol{v_{-1}}& 0 \\
0&\boldsymbol{v_{-1}}
\end{pmatrix},\quad\quad
\Xi_{\boldsymbol{ t_{-1}}}=
\begin{pmatrix}
\boldsymbol{t_{-1}}& 0 \\
0&\boldsymbol{t_{-1}}
\end{pmatrix}.\nonumber
\end{equation}
Let $A_1^{(i)}(1\leq i\leq 3)$ be the matrix expression of the operator $\mathcal{A}(\omega,\delta)$ under the following function. That is,
$$
\mathcal{A}(\omega,\delta)\boldsymbol{\Xi}_i=\boldsymbol{\Xi}_i A_1^{(i)},\quad i=1,2,3.
$$
To ensure that there exists a nontrivial kernel for the operator $\mathcal{A}(\omega,\delta)$, the determinant
of the matrix $A_1^{(i)}(1\leq i\leq3)$ should vanish.

$\bf{For \ i=1:}$ It follows from Lemma \ref{Spectrum1} and Lemma \ref{Spectrum2} that $\alpha_{30}=0$, $b_{10}=g_{3,0}(R)=0$, and
\begin{equation}\label{A1}
A_1^{(1)}=
\begin{pmatrix}
\tilde{\alpha}_{40} &-\alpha_{40}\\
-\frac{1}{2}+\tilde{b}_{20}  &-(\frac{1}{2}+b_{20})
\end{pmatrix}.%\nonumber
\end{equation}
Combining with asymptotic expansion \eqref{J0}, \eqref{H0}, we can obtain that
$$
\det A_1^{(1)}=\frac{R\tau^2}{2\mu}\left(\epsilon-\frac{1}{8\mu}\omega^2\rho R^2+\frac{1}{8\mu}\omega^2\epsilon\rho R^2\right)+\mathcal{O}\Big({\omega^2\epsilon\tau^4(1+\ln \tau+\ln \omega)}\Big).
$$
Here, if we only consider the first two items, i.e. $\epsilon$ and $\omega^2$ items, there holds that
$$
\frac{\omega^2\rho R^2}{8\mu}-\epsilon+\mathcal{O}\Big(\omega^2\epsilon+\omega^2\ln\omega\epsilon\Big)=0,
$$
which corresponds to the formula \eqref{formuladisk} in Lemma \ref{thmdiskomega} with $i=3$.

$\bf{For \ i=2:}$ From Lemma \ref{Spectrum1} and Lemma \ref{Spectrum2}, the $4\times 4$ matrix $A_1^{(2)}$ can also be given by
\begin{equation}\label{A2}
A_1^{(2)}=
\begin{pmatrix}
\tilde{\alpha}_{11} & -{\alpha}_{11}&  \tilde{\alpha}_{31}&  - {\alpha}_{31}\\
-\frac{1}{2}+\tilde{a}_{11}  &-(\frac{1}{2}+a_{11}) & \tilde{b}_{11}  & -{b}_{11}\\
\tilde{\alpha}_{21} & -{\alpha}_{21}&  \tilde{\alpha}_{41}&  - {\alpha}_{41}\\
\tilde{a}_{21}& -{a}_{21}&  -\frac{1}{2}+\tilde{b}_{21}&-(\frac{1}{2}+b_{21})
\end{pmatrix}.%\nonumber
\end{equation}
First, from equations \eqref{J1}, \eqref{H1}, there holds
$$
J_{1}(t)H_{1}(t)=-\frac{\mathrm{i}}{\pi}+\frac{\mathrm{i}t^2}{8\pi} \left(-1+2E_c+4\ln t\right)-\frac{\mathrm{i}t^4}{32\pi}\left(-\frac{10}{3}+2E_c+4\ln t\right)
+\mathcal{O}\left( t^6\ln t \right),
$$
$$
J_{1}^{\prime}(t) H_{1}^{\prime}(t)=\frac{\mathrm{i}}{\pi t^2}+\frac{\mathrm{i}}{8\pi} \left(-1+2E_c+4\ln t\right)-\frac{\mathrm{i}t^2}{32\pi}\left(-\frac{10}{3}+6E_c+12\ln t\right)
+\mathcal{O}\left(t^4\ln t \right),
$$
and
$$
J_{1}(t) H_{1}^{\prime}(t)=\frac{\mathrm{i}}{\pi t}+\frac{\mathrm{i}t}{8\pi} \left(1+2E_c+4\ln t\right)
-\frac{\mathrm{i}t^3}{32\pi}\left(-\frac{14}{3}+4E_c+8\ln t\right)
+\mathcal{O}\left(t^5\ln t \right),
$$
$$
J_{1}^{\prime}(t) H_{1}(t)=-\frac{\mathrm{i}}{\pi t}+\frac{\mathrm{i}t}{8\pi} \left(1+2E_c+4\ln t\right)
-\frac{\mathrm{i}t^3}{32\pi}\left(-\frac{14}{3}+4E_c+8\ln t\right)
+\mathcal{O}\left(t^5\ln t \right).
$$
Then, through the asymptotic expansions of each element in matrix $A_1^{(2)}$~(although tedious), we can derive that
\begin{equation}\label{detA2}
\epsilon=-\frac{R^2}{8}\left[k_s^2\left(E_c+2\ln({k_sR})\right)+k_p^2\left(E_c+2\ln({k_pR})\right)\right]+\mathcal{O}\left(\omega^4\ln \omega \right).
\end{equation}
In fact, by calculations, we can verify the following equation
\begin{align}
&\rho\omega^2\Big(\ln\omega+\ln{(\sqrt{\rho}\tau)}\Big){a}_{\lambda,\mu}\pi R^2 +\rho\omega^2\Big({b}_{\lambda,\mu}\pi R^2-\frac{R^2}{2}(\sigma_{1}\ln{R}-\frac{\sigma_2}{2})\Big) \nonumber \\
&=-\frac{R^2}{16}\Big((k_s^2+k_p^2)\left(2E_c+4\ln({\tau R})\right)+4\left(k_s^2\ln{k_s}+k_p^2\ln{k_p}\right)\Big),\nonumber
\end{align}
and
\begin{align}
\frac{\sigma_{1}\ln{R}-\frac{\sigma_2}{2}+2\pi \gamma_{\tau\omega}}{\sigma_{1}\ln{R}-\frac{\sigma_2}{2}+2\pi \gamma_{\omega}}=\frac{(k_s^2+k_p^2)\left(2E_c+4\ln({\tau R})\right)+4\left(k_s^2\ln{k_s}+k_p^2\ln{k_p}\right)}{(k_s^2+k_p^2)\left(2E_c+4\ln{ R}\right)+4\left(k_s^2\ln{k_s}+k_p^2\ln{k_p}\right)}.\nonumber
\end{align}
Hence, the equation \eqref{detA2} corresponds to the formula \eqref{formuladisk} in Lemma \ref{thmdiskomega} with $i=1,2$.

$\bf{For \ i=3:}$ Similarly, it follows from Lemma \ref{Spectrum1} and Lemma \ref{Spectrum2} that
\begin{equation}\label{A3}
A_1^{(3)}=
\begin{pmatrix}
\tilde{\alpha}_{1(-1)} & -{\alpha}_{1(-1)}&  \tilde{\alpha}_{3(-1)}&  - {\alpha}_{3(-1)}\\
-\frac{1}{2}+\tilde{a}_{1(-1)}  &-(\frac{1}{2}+a_{1(-1)}) & \tilde{b}_{1(-1)}  & -{b}_{1(-1)}\\
\tilde{\alpha}_{2(-1)} & -{\alpha}_{2(-1)}&  \tilde{\alpha}_{4(-1)}&  - {\alpha}_{4(-1)}\\
\tilde{a}_{2(-1)}& -{a}_{2(-1)}&  -\frac{1}{2}+\tilde{b}_{2(-1)}&-(\frac{1}{2}+b_{2(-1)})
\end{pmatrix}.%\nonumber
\end{equation}
Due to the following results
$$
{\alpha}_{1(-1)}={\alpha}_{11},\quad {\alpha}_{2(-1)}=-{\alpha}_{21},\quad {\alpha}_{3(-1)}=-{\alpha}_{31},\quad {\alpha}_{4(-1)}={\alpha}_{41},
$$
and
$$
a_{1(-1)}=a_{11},\quad a_{2(-1)}=-a_{21},\quad b_{1(-1)}=-b_{11},\quad b_{2(-1)}=b_{21}.
$$
Therefore, $\det A_1^{(2)}=\det A_1^{(3)}$, which also corresponds to the formula \eqref{formuladisk} in Lemma \ref{thmdiskomega} with $i=1,2$.

\section{Numerical illustrations}In this section, we conducted some numerical simulations to confirm the theoretical findings in the previous sections. First, we compared the two resonance frequencies with different values of $\delta$, namely the characteristic value $\omega_j^{(e)}$ of matrices $A_N^{(q)}(1\leq q\leq3)$, and the solution $\omega_j^{(c)}$ of equation \eqref{multiformula} in Theorem \ref{multifrequencies}. Here, $A_N^{(q)}(1\leq q\leq3)$ correspond to the matrix of operator $\mathcal{A}(\omega,\delta)$ defined in \eqref{block diagonal2} under the characteristic functions $\boldsymbol{\Xi}_i(1\leq i\leq3)$. In addition, we displayed how the elastic displacement field behaves when sub-wavelength resonance occurs. In this analysis, we only consider the resonance mode in multi-layer concentric disks.
\subsection{The resonance frequency}In this part, we will calculate and compare the resonance frequencies, denoted by $\omega_j^{(e)}$ and $\omega_j^{(c)}$ for some fixed $\delta\ll 1$ and $\epsilon\ll 1$. Due to the fact that the determinants of $A_N^{(2)}$ and $A_N^{(3)}$ have the same roots $\omega$, so for convenience, we only need to calculate $A_N^{(q)}$, $q=1,2$. In the following numerical experiments, the background parameters are $(\rho,\lambda,\mu):=(1,2,1)$, and the high-contrast parameters follow from \eqref{delta} associated with $\tau=1$. Let the radii of layers are equidistant, i.e. for $N$ resonators,
\begin{equation}\label{radii1}
\Gamma_j^{\pm}=\{|\mathbf{x}|=r_{j}^{\pm}\},\quad r_j^{+}=2-\frac{2(j-1)}{N},\quad r_j^{-}=2-\frac{2j-1}{N},\quad j=1,2,\dots,N,
\end{equation}
such that $r_1^{+}=2$. Furthermore, denote $f^{(q)}_N(\omega):=\text{det}(A_N^{(q)})(\omega,\delta)$, By using Muller's method, we compute the following root-finding problem
\begin{equation}
	f^{(q)}_N(\omega)=0,\quad q=1,2. \notag
\end{equation}

As shown in Table \ref{table:example 1}, Table \ref{table:example 2} and Table \ref{table:example 3}, there are 12 sub-wavelength resonance frequencies for $N =4$. The first four frequencies correspond to a single shear frequency~($q=1$), while the rest correspond to a double root primary frequency~($q=2$). By comparison, we can see that the frequencies, $\omega_j^{(e)}$ and $\omega_j^{(c)}$, only show very small differences, which confirms the effectiveness of asymptotic analysis. In fact, even if there is a singular $ \omega^2\ln \omega$ term in two dimensions, the real part of the eigenfrequencies still exhibits an approximate order of $\sqrt{\delta}$, just as in three dimensions. Moreover, the imaginary part of shear frequencies grow monotonically with the real part. In contrast, the primary frequencies display a non-monotonic trend, characterized by an initial decrease followed by an increase. This is significantly different from the three dimensional results. We attribute this behavior to the blow-up of the $\ln \omega$ term as $\omega \to 0$, which in turn requires a larger imaginary component to achieve balance. In addition, Table \ref{table:example 2} and Table \ref{table:example 3} provide the resonance frequencies for $\delta=10^{-4}$ and $\delta=10^{-6}$, respectively. The resonance frequency of $\delta=10^{-6}$ is almost one tenth of $\delta=10^{-4} $, which further confirms that $\omega$ can be scaled to $\sqrt{\delta}$. The frequencies $\omega_j^{(e)}$ and $\omega_j^{(c)}$ can become more consistent as $\delta$ decreases.

\begin{table}[H]
	\captionsetup{font=small}
	\caption{The resonance frequencies with $N=4$, the layers are chosen by \eqref{radii1} and $\delta=10^{-5}$.}
	\label{table:example 1}
	\centering
	\tiny
	\resizebox{0.7\textwidth}{!}{
		\begin{tabular}{ccc}%??c??????????
			\toprule%?????
			%&\multicolumn{2}{c}{\textbf{\underline{Resulrsummary}}}& \\%???;????;?????Resultsummary??\textbf{}?\underline{}??????????????
			$j$ & $\omega_j^{(e)}$&$\omega_j^{(c)}$  \\
			\midrule%?????
			$1$&$0.005517+0.000000014\text{i}$ &  $0.005529+0.000000017\text{i}$    \\
			$2$&$0.015223+0.000000097\text{i}$ &  $0.015213+0.000000096\text{i}$    \\
			$3$&$0.022261+0.000000872\text{i}$ &  $0.022249+0.000008073\text{i}$    \\
			$4$&$0.027129+0.000001858\text{i}$ &  $0.027125+0.000001850\text{i}$   \\
			$5$&$0.001552-0.000028774\text{i}$ &  $0.001554+0.000028371\text{i}$  \\
			$6$&$0.017528-0.000003617\text{i}$ &  $0.017527-0.000003621\text{i}$ \\
			$7$&$0.029591-0.000001392\text{i}$ &  $0.029581-0.000001390\text{i}$ \\
			$8$&$0.037349-0.000003518\text{i}$ &  $0.037354-0.000003589\text{i}$ \\
			\bottomrule%?????
		\end{tabular}
	}
\end{table}

\begin{table}[H]
	\captionsetup{font=small}
	\caption{The resonance frequencies with $N=4$, the layers are chosen by \eqref{radii1} and $\delta=10^{-4}$.}
	\label{table:example 2}
	\centering
	\tiny
	\resizebox{0.7\textwidth}{!}{
		\begin{tabular}{ccc}%??c??????????
			\toprule%?????
			%&\multicolumn{2}{c}{\textbf{\underline{Resulrsummary}}}& \\%???;????;?????Resultsummary??\textbf{}?\underline{}??????????????
			$j$ & $\omega_j^{(e)}$&$\omega_j^{(c)}$  \\
			\midrule%?????
			$1$&$0.017052+0.000000275\text{i}$ &  $0.017730+0.000000261\text{i}$     \\
			$2$&$0.049018+0.000032007\text{i}$ &  $0.049319+0.000032677\text{i}$    \\
			$3$&$0.071528+0.000281906\text{i}$ &  $0.071530+0.000282035\text{i}$    \\
			$4$&$0.086270+0.000560944\text{i}$ &  $0.086279+0.000560777\text{i}$   \\
			$5$&$0.005180-0.000101872\text{i}$ &  $0.005201-0.000102344\text{i}$  \\
			$6$&$0.055971-0.000023098\text{i}$ &  $0.055981-0.000023193\text{i}$ \\
			$7$&$0.093085-0.000008201\text{i}$ &  $0.093177-0.000008257\text{i}$ \\
			$8$&$0.011349-0.000023256\text{i}$ &  $0.011056-0.000023358\text{i}$ \\
			\bottomrule%?????
		\end{tabular}
	}
\end{table}

\begin{table}[H]
	\captionsetup{font=small}
	\caption{The resonance frequencies with $N=4$, the layers are chosen by \eqref{radii1} and $\delta=10^{-6}$.}
	\label{table:example 3}
	\centering
	\tiny
	\resizebox{0.7\textwidth}{!}{
		\begin{tabular}{ccc}%??c??????????
			\toprule%?????
			%&\multicolumn{2}{c}{\textbf{\underline{Resulrsummary}}}& \\%???;????;?????Resultsummary??\textbf{}?\underline{}??????????????
			$j$ & $\omega_j^{(e)}$&$\omega_j^{(c)}$  \\
			\midrule%?????
			$1$&$0.001740+0.000000021\text{i}$ &  $0.001740+0.000000023\text{i}$    \\
			$2$&$0.004815+0.000002974\text{i}$ &  $0.004815+0.000002974\text{i}$    \\
			$3$&$0.007028+0.000030042\text{i}$ &  $0.007027+0.000030040\text{i}$    \\
			$4$&$0.008569+0.000060162\text{i}$ &  $0.008571+0.000060163\text{i}$   \\
			$5$&$0.000470-0.000012744\text{i}$ &  $0.000467-0.000012744\text{i}$  \\
			$6$&$0.005568-0.000002617\text{i}$ &  $0.005570-0.000002619\text{i}$ \\
			$7$&$0.009327-0.000000793\text{i}$ &  $0.009327-0.000000792\text{i}$ \\
			$8$&$0.011889-0.000002691\text{i}$ &  $0.011889-0.000002690\text{i}$ \\
			\bottomrule%?????
		\end{tabular}
	}
\end{table}
%%%%求根图
\begin{figure}
	\includegraphics[width=1\textwidth]{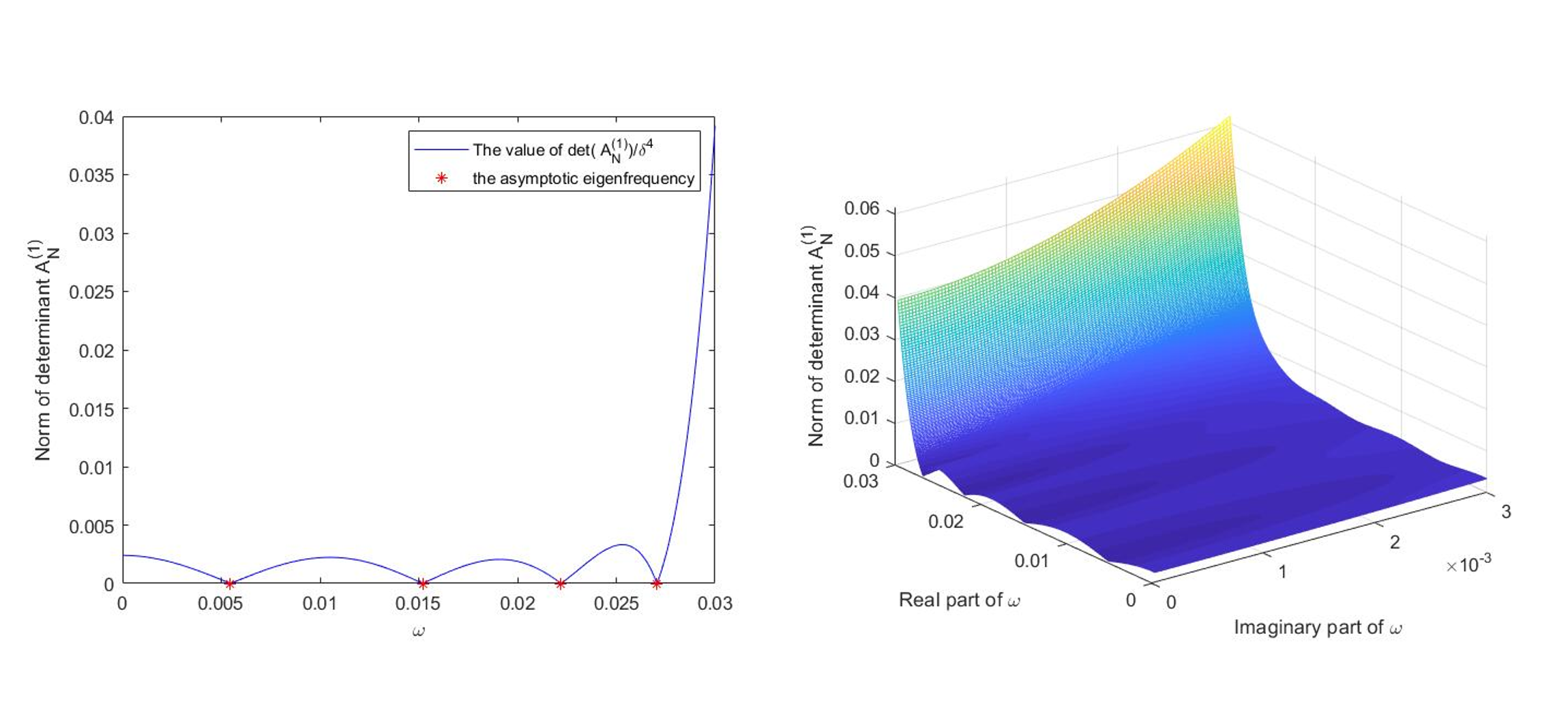}
	\caption{\label{fig:frequencies_1}The normalized determinants of $q=1$ in the setup of \eqref{radii1} with $N=4$.}
\end{figure}
\begin{figure}
	\includegraphics[width=1\textwidth]{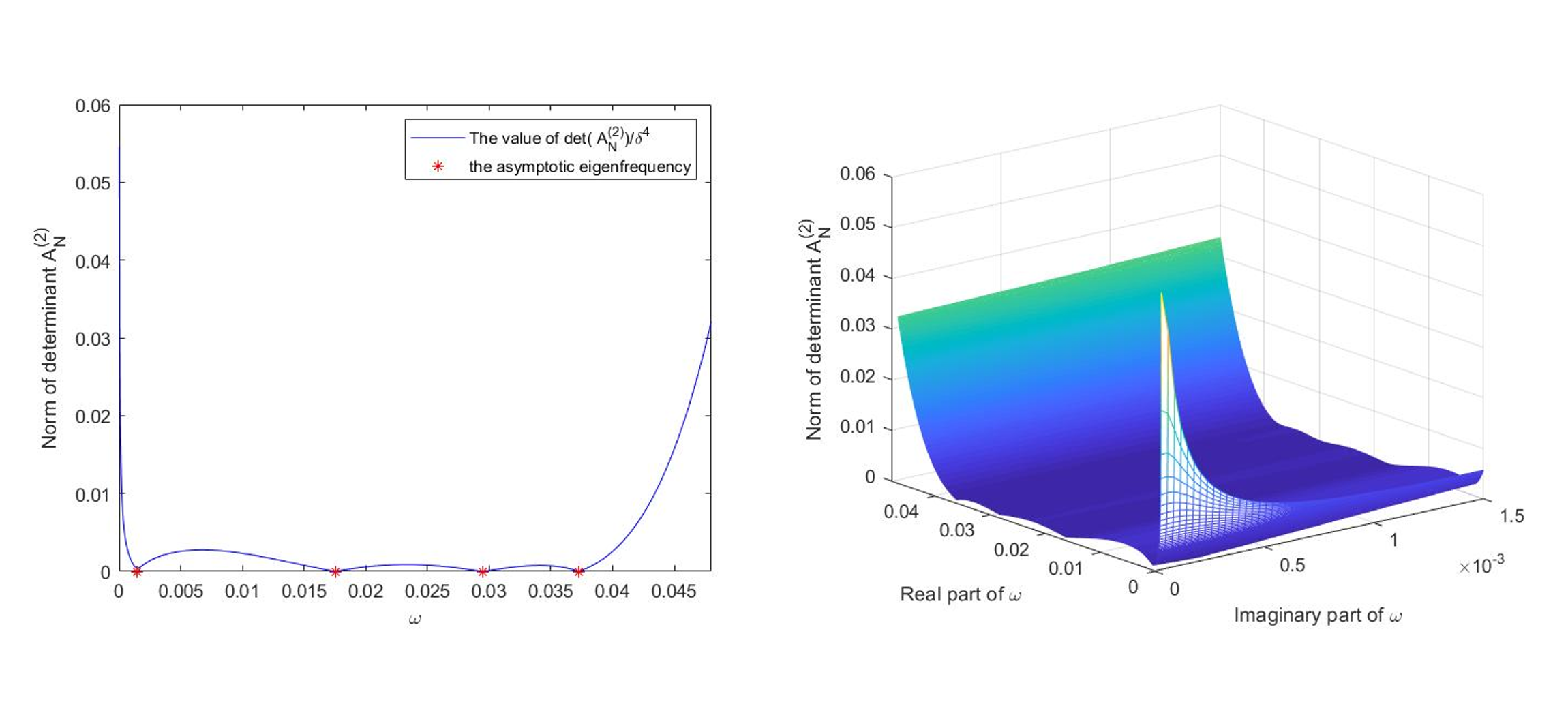}
	\caption{\label{fig:frequencies_2}The normalized determinants of $q=2$ in the setup of \eqref{radii1} with $N=4$.}
\end{figure}
To further illustrate that the obtained solution satisfies $f_N^{(q)}(\omega)=0$, let
\begin{equation}\label{normalized}
	\tilde{f}^{(q)}_N(\omega):=\frac{|f_N^{(q)}|(\omega)}{\delta^{\lfloor N+1\rfloor/2}},\quad q=1,2,
\end{equation}
with $\delta=10^{-5}$. In Fig \ref{fig:frequencies_1} and Fig \ref{fig:frequencies_2}, we provide graphs of the normalized determinant $\tilde{f}^{(1)}_N(\omega)$ and $\tilde{f}^{(2)}_N(\omega)$, respectively. The left panel describes the normalized determinant of $\operatorname{Re}(\omega)$ with $\operatorname{Im}(\omega)=0$, while the right panel displays both $\operatorname{Re}(\omega)$ and $\operatorname{Im}(\omega)$. In fact, Fig \ref{fig:frequencies_1} shows that a satisfactory approximation can be achieved even if the imaginary part is neglected. If the imaginary direction is not explored accurately enough, the fourth local minimum fails to reach the prescribed root-finding tolerance, which is consistent with the relatively large imaginary part in Table \ref{table:example 1}. From Fig \ref{fig:frequencies_2}, there holds that the first local minimum fails to meet the prescribed root-finding tolerance. Comparing the two right panels, one can see that $\tilde{f}_N^{(2)}\gg \tilde{f}_N^{(1)}$ as $\omega\to 0$. The reason is that the growth of the singular term $\ln\omega $ leads to a relatively large $\tilde{f}_N^{(2)}$, while the shear frequency lacks this term.
%%%%范数图
\begin{figure}
	\includegraphics[width=1\textwidth]{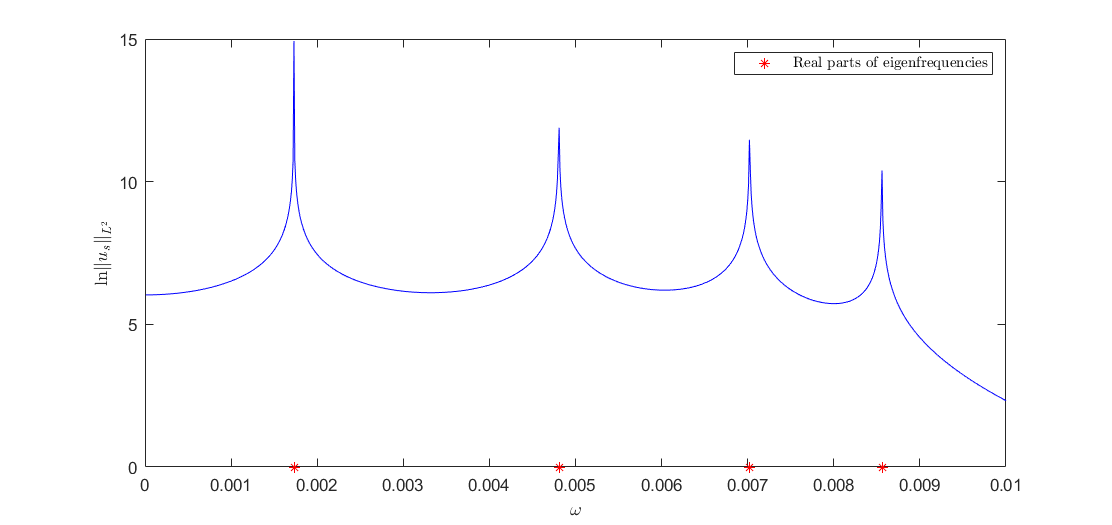}
	\caption{\label{fig:s_modes}Norm of the displacement fields $\mathbf{u}_{S}$ in the setup of \eqref{radii1} with $N=4$.}
\end{figure}

\begin{figure}
	\includegraphics[width=1\textwidth]{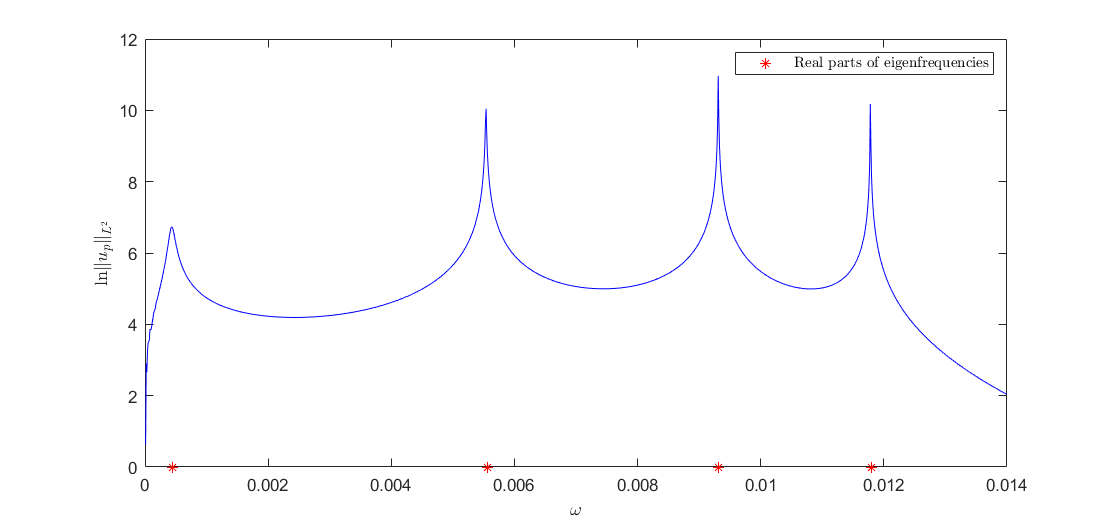}
	\caption{\label{fig:p_modes}Norm of the displacement fields $\mathbf{u}_{P}$ in the setup of \eqref{radii1} with $N=4$.}
\end{figure}
\begin{figure}%%%%eigenmodes图
	\includegraphics[width=1\textwidth]{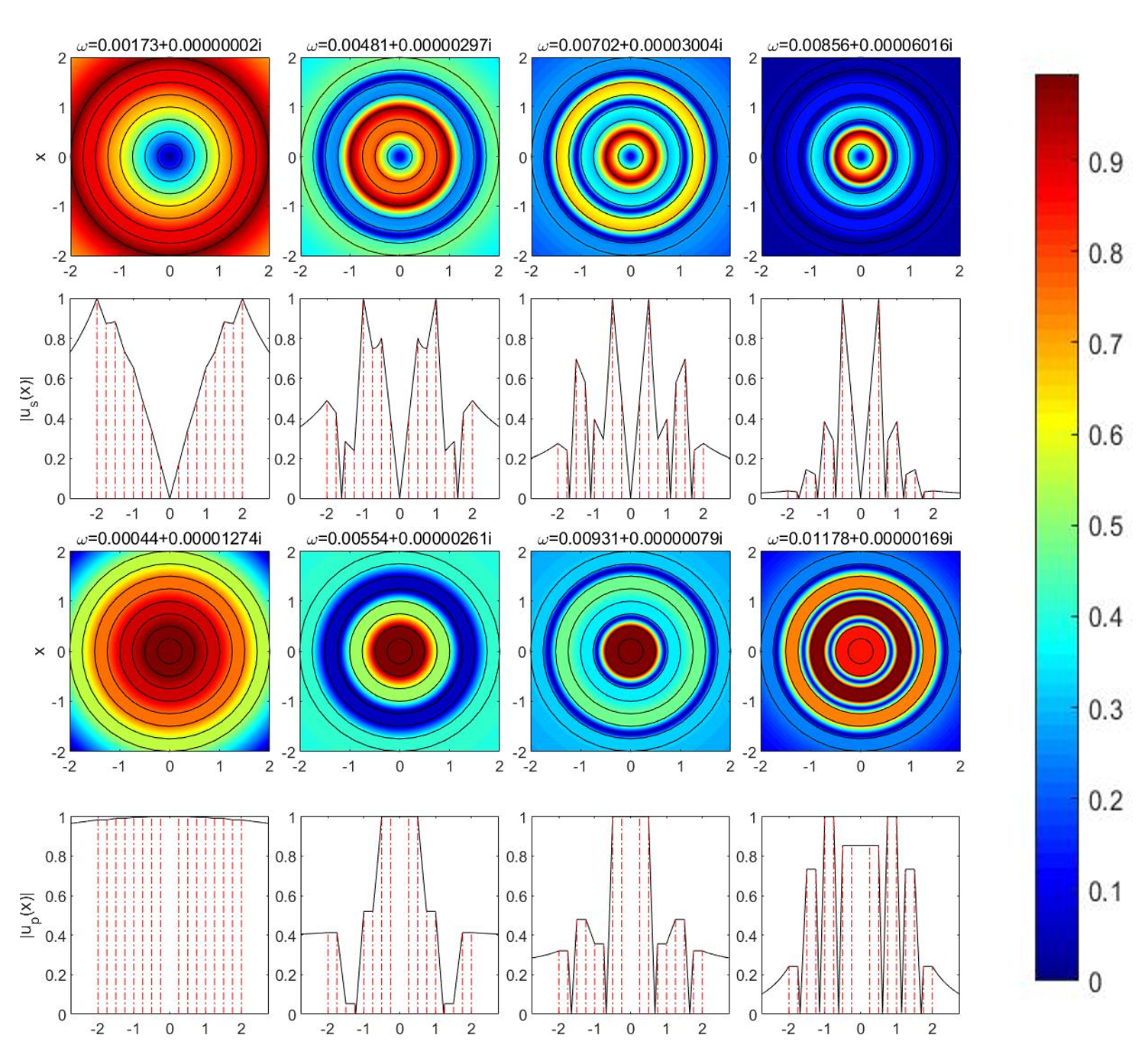}
	\caption{\label{fig:resonant_modes}{The elastic eigenmodes in the setup of \eqref{radii1} with $N=4$. The upper shows the displacement field corresponding to the shear waves, while the below shows primary waves}.}
\end{figure}
\subsection{The resonance model}
In this subsection, we present the displacement fields for equation \eqref{Lame system2} under the excitation of plane wave
\begin{equation}\label{incident wave}
	\mathbf{u}^i=\mathbf{u}^i_{s}+\mathbf{u}^i_{p}=\mathbf{q}e^{\mathrm{i}k_s\mathbf{x}\cdot \mathbf{d}}+\mathbf{d}e^{\mathrm{i}k_p\mathbf{x}\cdot \mathbf{d}},
\end{equation}
where $\mathbf{d}\in \mathbb{S}^2$ is the direction of incidence, and $\mathbf{q}\in \mathbb{S}^2$ is any vector orthogonal to $\mathbf{d}$, i.e. $\mathbf{q}\cdot \mathbf{d}=0$.
By using the vector version of the Jacobi-Anger expansion\cite{ColtonKress}, namely,
\begin{equation}
	e^{\mathrm{i}k|\mathbf{x}|\cos\theta}=\sum_{m\in \mathbb{Z}}\mathrm{i}^mJ_m(k|\mathbf{x}|)e^{\mathrm{i}m\theta}.
\end{equation}
In order to facilitate calculations using the discrete matrix of $\mathcal{A}(\omega,\delta)$, we decompose $\mathbf{u}^i$ into its modal components. Here, $\mathbf{d}=(1,0)^{T}$, $\delta=10^{-6}$, and the parameters are the same as those chosen in Subsection 5.1. The results, in Fig \ref{fig:resonant_modes}, show that the displacement distribution maintains the symmetry of the nested resonator while exhibiting increasingly complex oscillation structures. Besides, for primary frequencies, the displacement distribution within each resonator remains approximately constant, supporting the point scatterer approximation derived in Theorem \ref{pointscattererapproximation}.

Finally, we decompose the total displacement field into its shear and principal components. In Figure \ref{fig:s_modes} and Figure \ref{fig:p_modes}, we respectively demonstrate how the $L^2$ norm of the shear wave $\mathbf{u}_{S}$ and the primary wave $\mathbf{u}_{P}$ defined in equation \eqref{Lame system2} varies with $\omega$ for $N=4$. The results indicate that the peak in the displacement field norm occurs when $\omega$ approaches the real part of the eigenfrequency, which is consistent with the theoretical results of Theorem \ref{multifrequencies}.

\section*{Acknowledgments}
The work of H. Liu is supported by NSFC/RGC Joint Research Scheme, N CityU101/21, ANR/RGC Joint Research Scheme, A-CityU203/19, and the Hong Kong RGC General Research Funds (projects 11311122, 11304224 and 11300821). The work of Y. Jiang is supported by the China Natural National Science
Foundation (No. 123B2017).

\end{document}